\documentclass[reqno,10pt,centertags]{amsart}
\usepackage{amsmath,amsthm,amscd,amssymb,latexsym,upref,enumerate}
\usepackage{bbm}
\usepackage{nicefrac}
\usepackage{color}
\usepackage{hyperref}
\newcommand*{\mailto}[1]{\href{mailto:#1}{\nolinkurl{#1}}}
\newcommand{\arxiv}[1]{\href{http://arxiv.org/abs/#1}{arXiv:#1}}

\newcommand{\N}{{\mathbb N}}

\newcommand{\bbC}{{\mathbb{C}}}

\newcommand{\bbN}{{\mathbb{N}}}

\newcommand{\bbR}{{\mathbb{R}}}
\newcommand{\bbT}{{\mathbb{T}}}

\newcommand{\cA}{{\mathcal A}}
\newcommand{\cB}{{\mathcal B}}

\newcommand{\cH}{{\mathcal H}}

   \def\sH{{\mathfrak H}}

\def\st{{\mathfrak t}}
\def\sss{{\mathfrak s}}

\DeclareMathOperator{\rank}{rank}
\DeclareMathOperator{\ran}{ran}
\DeclareMathOperator{\dom}{dom}
\DeclareMathOperator{\mul}{mul}
\DeclareMathOperator{\ls}{lin.span}

\newcommand{\no}{\notag}
\newcommand{\lb}{\label}
\newcommand{\f}{\frac}

\newcommand{\ol}{\overline}
\newcommand{\bs}{\backslash}

\newcommand{\wti}{\widetilde}

\newcommand{\oh}{o}
\newcommand{\hatt}{\widehat}
\newcommand{\dott}{\,\cdot\,}

\renewcommand{\dot}{\overset{\textbf{\Large.}}}

\newcommand{\bi}{\bibitem}

\let\geq\geqslant
\let\leq\leqslant

\newcommand{\la}{\lambda}
\newcommand{\lam}{\lambda}

\newcommand{\al}{\alpha}
\newcommand{\be}{\beta}

\newcommand{\Lr}{{L^2((a,b);r\,dx)}}

\makeatletter
\def\theequation{\@arabic\c@equation}

\allowdisplaybreaks
\numberwithin{equation}{section}

\newtheorem{theorem}{Theorem}[section]
\newtheorem{proposition}[theorem]{Proposition}
\newtheorem{lemma}[theorem]{Lemma}

\newtheorem{definition}[theorem]{Definition}
\newtheorem{hypothesis}[theorem]{Hypothesis}

\theoremstyle{remark}
\newtheorem{remark}[theorem]{Remark}

\begin{document}

\title[Sesquilinear Forms for (Singular) Sturm--Liouville Operators]{On Sesquilinear Forms for Lower Semibounded
(Singular) Sturm--Liouville Operators}

\author[J. Behrndt]{Jussi Behrndt}
\address{Technische Universit\"{a}t Graz\\
Institut f\"ur Angewandte Mathematik\\
Steyrergasse 30\\
8010 Graz, Austria}
\email{\mailto{behrndt@tugraz.at}}
\urladdr{\url{https://www.math.tugraz.at/~behrndt/}}

\author[F.\ Gesztesy]{Fritz Gesztesy}
\address{Department of Mathematics,
Baylor University, Sid Richardson Bldg., 1410 S.\,4th Street, Waco, TX 76706, USA}
\email{\mailto{Fritz\_Gesztesy@baylor.edu}}
\urladdr{\url{http://www.baylor.edu/math/index.php?id=935340}}

\author[S. Hassi]{Seppo Hassi}
\address{Department of Mathematics and Statistics, University of Vaasa, P.O. Box 700, 65101 Vaasa, Finland}
\email{\mailto{sha@uwasa.fi}}
\urladdr{\url{https://lipas.uwasa.fi/~sha/}}

\author[R.\ Nichols]{Roger Nichols}
\address{Department of Mathematics (Dept.~6956), The University of Tennessee at Chattanooga,
615 McCallie Avenue, Chattanooga, TN 37403, USA}
\email{\mailto{Roger-Nichols@utc.edu}}
\urladdr{\url{https://sites.google.com/mocs.utc.edu/rogernicholshomepage/home}}

\author[H.S.V. de Snoo]{Henk de Snoo}
\address{Bernoulli Institute for Mathematics, Computer Science and Artificial Intelligence,
University of Groningen, P.O. Box 407, 9700 AK Groningen, Netherlands}
\email{\mailto{h.s.v.de.snoo@rug.nl}}

\date{\today}
\@namedef{subjclassname@2020}{\textup{2020} Mathematics Subject Classification}


\begin{abstract}
Any self-adjoint extension of a (singular) Sturm--Liouville operator bounded from below
uniquely leads to an associated sesquilinear form.
This form is characterized in terms of principal and nonprincipal solutions
of the Sturm--Liouville operator by using generalized boundary values. We provide these forms in detail
in all possible cases (explicitly, when both endpoints are limit circle, when one endpoint is limit circle,
and when both endpoints are limit point).
\end{abstract}

\maketitle

{\scriptsize{\tableofcontents}}

\maketitle



\section{Introduction} \lb{s1}

The traditional three-coefficient Sturm--Liouville (generalized eigenvalue) problem on an arbitrary open interval
$(a,b) \subseteq \bbR$ is of the form
\begin{equation} \lb{1.1}
 -(p(x)f'(x))'+q(x) f(x) = z r(x) f(x) \, \text{ for a.e.~$x \in (a,b)$}, \; z \in \bbC,
\end{equation}
where the coefficients $p, q, r$ are real-valued (Lebesgue) a.e.~on $(a,b)$, $p, r > 0$ a.e.~on $(a,b)$, and $1/p, q, r \in L^1_{loc}((a,b); dx)$. In addition, $z \in \bbC$ represents a (generally, complex-valued) spectral parameter, and $f$ and $pf'$ are assumed to be locally absolutely continuous on $(a,b)$; see Section \ref{s2} for details. More precisely, the differential expression $\tau$ underlying \eqref{1.1},
\begin{equation}
\tau = \f{1}{r(x)} \bigg[ - \f{d}{dx} p(x) \f{d}{dx} + q(x)\bigg] \, \text{  for a.e.~$x \in (a,b)$},    \lb{1.2}
\end{equation}
naturally leads to a minimal closed symmetric operator
$T_{min}$ in the Hilbert space $L^2((a,b); r\,dx)$ (cf.\ \eqref{2.10} and \eqref{2.11}) and its
deficiency indices are then given by $(0,0)$, $(1,1)$, or $(2,2)$.
From the outset, the operator $T_{min}$ is in general not lower semibounded.
However, in this paper it will be assumed that equation \eqref{1.1}
has solutions which are nonoscillatory at the endpoints $a$ and $b$ for some
$z \in \bbR$ and in this case $T_{min}$ turns out to be lower semibounded.
As a consequence, all self-adjoint extensions of $T_{min}$ are then
lower semibounded, see Proposition \ref{23.t2.5}. For example, in the special case of  a one-dimensional Schr\"odinger operator where $\tau$ simplifies to
$\tau= -(d^2/dx^2)+q(x)$ for a.e.~$x \in (a,b)$,
quantum mechanical considerations typically lead to the requirement of lower semibounded self-adjoint extensions of $T_{min}$
and the characterization of the underlying quadratic forms (representing the sum of kinetic and potential energy)
corresponding to them.

In this paper we consider the natural and nontrivial question of determining the form domains associated with general, that is, lower semibounded, self-adjoint, singular, three-coefficient
Sturm--Liouville operators associated with $L^2((a,b); r\,dx)$-realizations of the differential expression $\tau$ in \eqref{1.2}.
The corresponding sesquilinear forms are then connected to integrals of the form
\begin{equation}\lb{1.3}
 \int_a^b\, dx\, \big[p(x) \ol{f'(x)} g'(x)  + q(x)\ol{f(x)}g(x)\big]
\end{equation}
for ``appropriate'' elements $f, g \in L^2((a,b); r\,dx)$.
However, if  one of the functions $f$ or $g$ is not compactly supported in $(a,b)$,
there might well be a problem with the convergence of the integral in \eqref{1.3}.
This problem will be avoided when rewriting the integral
by means of the nonoscillatory solutions of \eqref{1.1} mentioned above.
These solutions will also be used to introduce generalized
boundary values (see Proposition \ref{23.t7.3.11}) that are associated to the particular self-adjoint extension of $T_{min}$ under consideration. The main results in this paper are formulated  in terms of
proper interpretations of the integral \eqref{1.3} and, in particular,
in terms of generalized boundary values, see Proposition \ref{23.t7.3.11}.

The history of Sturm--Liouville problems, and, especially, the naturally associated spectral theory, is incredibly rich. Hence, we can only point to some of the classical contributions by
Weyl \cite{We09}--\cite{We50}, Titchmarsh \cite{Ti41}--\cite{Ti45}, \cite[Chs.~I--VI]{Ti62}, and Kodaira \cite{Ko49}, \cite{Ko50}, and, for more recent accounts, refer to the monographs
\cite[Sects.~127, 132]{AG81}, 
\cite[Ch.~6]{BHS20},
\cite[Chs.~4, 6--8]{BBW20},
\cite[Ch.~9]{CL85},
\cite[Sect.~13.6, 13.9, 13.10]{DS88},
\cite[Ch.~2]{EK82},
\cite[Sect.~3.10]{EE18},
\cite[Chs.~4--10, 13]{GNZ24}, 
\cite{He67}, 
\cite[Parts~II, III]{Jo62},
\cite[Ch.~III]{JR76},
\cite[Sect.~11.9]{Lu22},
\cite[Sect.~15-19]{Na68},
\cite[Ch.~6]{Pe88},
\cite[Chs.~1--4, 6]{RK05},
\cite[Ch.~15]{Sc12},
\cite[Ch.~9]{Te14},
\cite[Sects.~3--7]{We87},
\cite[Ch.~13]{We03},
\cite[Sect.~8.4]{We80},
\cite[Ch.~5]{Yo91}, 
and
\cite[Chs.~7--10]{Ze05}.

The material in this paper is presented in a systematic and straightforward way.
A brief review of Sturm--Liouville theory is given in Section \ref{s2}.
The description of all self-adjoint extensions by means of
\textit{generalized boundary values} can be found in
Propositions \ref{23.t7.t3.12}, \ref{p4.1}, and \ref{p2.14},
depending on the endpoints being in the limit circle case or
in the limit point case. Section \ref{s2} also briefly surveys the history of the notion
of generalized boundary values (cf.\ Remark~\ref{r2.15}).  
In each of our principal Sections \ref{s3}, \ref{s4}, and \ref{s5},
one can find a systematic description of the quadratic forms
corresponding to the self-adjoint extensions in the various cases;
see Theorems \ref{case0}, \ref{case1}, \ref{caselp},
and \ref{caselplp}. 
The results are obtained via integration by parts of 
the expression $(f,T_{max}g)$ for $f, g \in \dom (T_{max})$, where $T_{max}$ denotes the maximal operator associated to \eqref{1.2}; see Lemma \ref{l3.3c} and Lemma \ref{l4.4}.  
This yields an alternative and very explicit formulation of the results in \cite[Ch.~6]{BHS20} in terms of generalized boundary values. These results generalize those of \cite[Sect.~4.5]{GNZ24} in the special case where $\tau$ is regular at $a$ and $b$.
The presentation is for the most part self-contained.
 For completeness and convenience of the reader,  we identify in Appendix \ref{{sA}}
the boundary triplet and the boundary pair used in  \cite{BHS20} to obtain 
the general formulation of the main results in Sections \ref{s3} and \ref{s4}.
In the appendix the emphasis is on the abstract analogue of  Lemma \ref{l3.3c} and Lemma \ref{l4.4}.
 The abstract results also lead to a description of  the Friedrichs extension in each of these sections by means of a boundary pair. 
 
\medskip

We conclude this introduction by briefly commenting on some of the notation employed in the bulk of this paper: The inner product in a separable (complex) Hilbert space $\cH$ is denoted by $(\,\cdot\,,\,\cdot\,)_{\cH}$ and is assumed to be linear with respect to the second argument.   If $T$ is a linear operator mapping (a subspace of) a Hilbert space into another, then $\dom(T)$, $\ran(T)$, and $\ker(T)$ denote the domain, range, and kernel (i.e., null space) of $T$, respectively. The analogous conventions are used for linear relations and sesquilinear forms (when applicable); in particular, the multi-valued part of a linear relation $T$ is denoted by $\mul(T)$.  Finally, $SL(2,\bbR)$ denotes the set of all $2\times 2$ matrices with real-valued entries and determinant one.

\section{Sturm--Liouville Operators, Generalized Boundary Values, and Self-Adjoint Realizations} \lb{s2}

The following hypothesis will be assumed throughout this paper.

\begin{hypothesis} \lb{h2.1}
Let $-\infty\leq a<b\leq \infty$.  Suppose that $p$, $q$, and $r$ are Lebesgue measurable on $(a,b)$ with $p^{-1}, q, r\in L^1_{loc}((a,b); dx)$ and real-valued a.e.~on $(a,b)$ with $r>0$ and $p>0$  a.e.~on $(a,b)$.
\end{hypothesis}

We recall the basic construction and properties of Sturm--Liouville differential expressions and their associated operators.  For a full treatment with proofs of the assertions in this section, we refer to \cite[Chapter 5]{GNZ24}.

Assuming Hypothesis \ref{h2.1}, we introduce the set
\begin{equation}\lb{2.1}
\mathfrak{D}_{\tau}((a,b))=\big\{g\in AC_{loc}((a,b))\, \big|\, g^{[1]}=p g' \in AC_{loc}((a,b))\big\}
\end{equation}
and the differential expression $\tau$ defined by
\begin{equation}
 \tau f = \frac{1}{r} \Big[ - \big(f^{[1]}\big)' + qf\Big] \in L_{loc}^1((a,b);r\,dx),\quad f\in \mathfrak{D}_{\tau}((a,b)),  \lb{2.2}
\end{equation}
where the expression
\begin{equation}
f^{[1]}=p f', \quad f \in \mathfrak{D}_{\tau}((a,b)),    \lb{2.3}
\end{equation}
is the \textit{first quasi-derivative} of $f$.  For each $f,g\in \mathfrak{D}_{\tau}((a,b))$, the (modified) Wronskian of $f$ and $g$ is defined by
\begin{equation}\lb{2.4}
W(f,g)(x) = f(x)g^{[1]}(x) - f^{[1]}(x)g(x),\quad x\in (a,b).
\end{equation}
Hence, $W(f,g)$ is locally absolutely continuous on $(a,b)$ and its derivative is
\begin{equation}
W(f,g)'(x) = \big[g(x) (\tau f)(x) - f(x) (\tau g)(x)\big] r(x) \, \text{ for a.e.~$x\in(a,b)$.}
\end{equation}
In particular, if $z\in \bbC$, then the Wronskian of two solutions $u_j(z,\,\cdot\,)\in \mathfrak{D}_{\tau}((a,b))$, $j\in\{1,2\}$, of $\tau u = z u$ on $(a,b)$ is constant.  Moreover, $W(u_1(z,\,\cdot\,),u_2(z,\,\cdot\,))\neq 0$ if and only if $u_1(z,\,\cdot\,)$ and $u_2(z,\,\cdot\,)$ are linearly independent.

\begin{definition} \lb{d2.2}
The differential expression $\tau$ is said to be \textit{regular} on $(a,b)$ if $-\infty<a<b<\infty$ $($i.e., $a$ and $b$ are finite$)$ and $p^{-1}, q, r\in L^1((a,b); dx)$; otherwise, $\tau$ is said to be \textit{singular} on $(a,b)$.
\end{definition}

If $\tau$ is regular on $(a,b)$, then for each $f\in \mathfrak{D}_{\tau}((a,b))$ the following limits exist and are finite:
\begin{equation}\lb{2.6}
\begin{split}
&f(a):=\lim_{x\downarrow a}f(x),\quad f^{[1]}(a):=\lim_{x\downarrow a}f^{[1]}(x),\\
&f(b):=\lim_{x\uparrow b}f(x),\quad f^{[1]}(b):=\lim_{x\uparrow b}f^{[1]}(x).
\end{split}
\end{equation}

The differential expression $\tau$ gives rise to linear operators in the Hilbert space $\Lr$ equipped with the standard inner product
\begin{equation}
(f,g)_{L^2((a,b); r\,dx)} = \int_a^b r(x)\, dx\, \overline{f(x)}g(x),\quad f,g\in \Lr.
\end{equation}
The \textit{maximal operator} associated to $\tau$ is denoted by $T_{max}$ and is defined by
\begin{align}
&T_{max}f = \tau f,\lb{2.8}\\
&f\in \dom(T_{max})=\big\{g\in \Lr\,\big|\, g\in \mathfrak{D}_{\tau}((a,b)),\, \tau g\in \Lr\big\}.\no
\end{align}
Furthermore, the Wronskian of any two functions $f,g\in \dom(T_{max})$ possesses finite boundary values at the endpoints of $(a,b)$; that is, the following limits exist and are finite:
\begin{equation}\lb{2.9}
W(f,g)(a) := \lim_{x\downarrow a}W(f,g)(x),\quad W(f,g)(b) := \lim_{x\uparrow b}W(f,g)(x).
\end{equation}
The \textit{pre-minimal operator} associated to $\tau$ is denoted by $\dot T$ and is defined by
\begin{align}
& \dot T f = \tau f,\lb{2.10}\\
& f \in \dom\big(\dot T\big) = \big\{ g\in\dom(T_{max}) \,|\, g \text{ has compact support in }(a,b)\big\}.\no
\end{align}
One can show that the operator $\dot T$ is densely defined and symmetric in the Hilbert space $\Lr$ and $\big(\dot T\big)^*=T_{max}$.  The \textit{minimal operator} associated to $\tau$ is denoted by $T_{min}$ and is defined to be the closure of the pre-minimal operator:
\begin{equation}
T_{min} := \overline{\dot T}.     \lb{2.11}
\end{equation}
In addition, $T_{min}$ and $T_{max}$ are adjoint to one another:
\begin{equation}\lb{2.12}
T_{min}^*=T_{max}\quad \text{and}\quad T_{max}^*=T_{min}.
\end{equation}

\begin{definition}
A measurable function $f:(a,b)\to \bbC$ is in $\Lr$ near $a$ $($resp., $b$$)$ if $\chi_{(a,c)}f$ $($resp., $\chi_{(c,b)}f$$)$ belongs to $\Lr$ for some $c\in (a,b)$.
\end{definition}

\begin{proposition}[Weyl's Alternative]\lb{t2.3}
Assume Hypothesis \ref{h2.1}. Then the following alternative holds: Either \\[1mm]
$(i)$ for every $z\in\bbC$, all solutions $u$ of $\tau u = z u$ are in $\Lr$ near $b$
$($resp., near $a$$)$, \\[1mm]
or, \\[1mm]
$(ii)$  for every $z\in\bbC$, there exists at least one solution $u$ of $\tau u = z u$ which is not in $\Lr$ near $b$ $($resp., near $a$$)$. In this case, for each $z\in\bbC\bs\bbR$, there exists precisely one solution $\psi_b$ $($resp., $\psi_a$$)$ of $\tau u = z u$ $($up to constant multiples$)$ which lies in $\Lr$ near $b$ $($resp., near $a$$)$.
\end{proposition}

\begin{definition} \lb{d4.10} Assume Hypothesis \ref{h2.1}.  In case $(i)$ in Proposition \ref{t2.3}, $\tau$ is said to be in the \textit{limit circle case} at $b$ $($resp., $a$$)$.  In case $(ii)$ in Proposition \ref{t2.3}, $\tau$ is said to be in the \textit{limit point case} at $b$ $($resp., $a$$)$.
\end{definition}

\begin{remark}
If $\tau$ is in the limit circle case at $b$ $($resp., $a$$)$, then $\tau$ is frequently called \textit{quasi-regular} at $b$ $($resp., $a$$)$.  If $\tau$ is in the limit circle case at both $a$ and $b$, then $\tau$ is frequently also called \textit{quasi-regular}.\hfill$\diamond$
\end{remark}

We recall that $T_{min}$ is \textit{lower semibounded} or \textit{bounded from below by $\lambda_0$}, and one writes $T_{min} \geq \lambda_0 I_{\Lr}$ (in this case, $\lambda_0$ is called a \textit{lower bound} of $T_{min}$), if
\begin{equation}\lb{2.13j}
(u,T_{min}u)_{\Lr}\geq \lambda_0(u,u)_{\Lr}, \quad u \in \dom(T_{min}).
\end{equation}
In particular, the \textit{lower bound of $T_{min}$} is the largest of all the lower bounds $\lambda_0$ for which \eqref{2.13j} holds.

The lower semiboundedness property of $T_{min}$ (equivalently, $\dot T$) is connected to the existence of distinguished nonoscillatory solutions, the so-called \textit{principal} and \textit{nonprincipal} solutions, at the endpoints $a$ and $b$.

\begin{definition} \lb{23.d2.4}
Assume Hypothesis \ref{h2.1} and fix $c\in (a,b)$ and $\lambda\in\bbR$. The differential expression $\tau - \lam$ is called {\it nonoscillatory at} $a$ $($resp., $b$$)$, if there exists a real-valued solution $u(\lambda,\dott)$ of
$\tau u = \lambda u$ that has finitely many zeros in $(a,c)$ $($resp., $(c,b)$$)$. Otherwise, $\tau - \lam$ is called {\it oscillatory at} $a$ $($resp., $b$$)$.  If $\tau - \lam$ is nonoscillatory at $a$ and $b$, one calls $\tau - \lam$ {\it nonoscillatory on $(a,b)$}. In addition, $\tau - \lambda$ is called {\it oscillatory on $(a,b)$} if it is oscillatory at least at one of the endpoints $a$ or $b$.
\end{definition}

\begin{proposition} \lb{23.t2.5}
Assume Hypothesis \ref{h2.1} and let $\lambda_0\in \bbR$. Then the following items $(i)$--\,$(iii)$ are
equivalent\,$:$ \\[1mm]
$(i)$ $T_{min}$ is bounded from below by $\lambda_0$; that is, $T_{min}\geq \lambda_0I_{\Lr}$.
\\[1mm]
$(ii)$ For all $\lambda \leq \la_0$, $\tau - \lam$ is nonoscillatory at $a$ and $b$.
\\[1mm]
$(iii)$ For all $\lambda\leq\la_0$, $\tau u = \lambda u$ has, for some $c_0,d_0\in (a,b)$, real-valued nonvanishing solutions $u_a(\lambda,\dott)$ and $\hatt u_a(\lambda,\dott)$ in the interval $(a,c_0]$, and real-valued nonvanishing solutions
$u_b(\lambda,\dott)$ and $\hatt u_b(\lambda,\dott)$ in the interval $[d_0,b)$, such that
\begin{align}
&W(u_a (\lambda,\dott),\hatt u_a (\lambda,\dott)) = 1,
\quad u_a (\lambda,x)=\oh(\hatt u_a (\lambda,x))
\text{ as $x\downarrow a$,} \lb{2.14} \\
&W(u_b (\lambda,\dott),\hatt u_b (\lambda,\dott))\, = 1,
\quad u_b (\lambda,x)\,=\oh(\hatt u_b (\lambda,x))
\text{ as $x\uparrow b$,} \lb{2.15}
\end{align}
and for all $c\in (a,c_0]$ and $d\in [d_0,b)$,
\begin{align}
&\int_a^c dx \, p(x)^{-1}u_a(\lambda,x)^{-2}=\int_d^b dx \,
p(x)^{-1}u_b(\lambda,x)^{-2}=\infty,  \lb{2.16} \\
&\int_a^c dx \, p(x)^{-1}{\hatt u_a(\lambda,x)}^{-2}<\infty, \quad
\int_d^b dx \, p(x)^{-1}{\hatt u_b(\lambda,x)}^{-2}<\infty. \lb{2.17}
\end{align}
\end{proposition}

\begin{definition}
Assume Hypothesis \ref{h2.1}, suppose that $T_{min}$ is bounded from below by $\lambda_0\in \bbR$ and let $\lambda\leq \lambda_0$. Then
$u_a(\lambda,\dott)$ $($resp., $u_b(\lambda,\dott)$$)$ in Proposition
\ref{23.t2.5}\,$(iii)$ is called a {\it principal} $($or {\it minimal}\,$)$
solution of $\tau u=\lambda u$ at $a$ $($resp., $b$$)$.
A real-valued solution ${\hatt u}_a(\lambda,\dott)$ $($resp., ${\hatt u}_b(\lambda,\dott)$$)$ of $\tau u=\lambda u$ linearly independent of $u_a(\lambda,\dott)$ $($resp., $u_b(\lambda,\dott)$$)$ is called a {\it nonprincipal} solution of $\tau u = \lambda u$ at $a$ $($resp., $b$$)$.
\end{definition}

Following \cite{GLN20} and \cite[Sect.~13.4]{GNZ24}, the next result introduces generalized boundary values at the endpoints $a$ and $b$ for functions belonging to $\dom(T_{max})$.

\begin{proposition}[Generalized boundary values]\lb{23.t7.3.11}
Assume Hypothesis \ref{h2.1} and let $\tau$ be in the limit circle case at $a$ and $b$ $($i.e., $\tau$ is quasi-regular on $(a,b)$$)$. In addition, assume that $T_{min} \geq \lambda_0 I_{\Lr}$ for some $\lambda_0 \in \bbR$, and denote by $u_t(\lambda_0,\dott)$ and $\widehat u_t(\lambda_0,\dott)$ principal and nonprincipal solutions of $\tau u = \lambda_0 u$ on $(a,b)$, respectively, at $t\in \{a,b\}$ that satisfy
\begin{equation}
W(\hatt u_a(\lambda_0,\dott), u_a(\lambda_0,\dott))
= W(\hatt u_b(\lambda_0,\dott), u_b(\lambda_0,\dott)) = 1.
\lb{7.3.33AB}
\end{equation}
Introducing $v_j \in \dom(T_{max})$, $j=1,2$, via
\begin{align}
v_1(x) = \begin{cases} \hatt u_a(\lambda_0,x), & \text{for $x$ near a}, \\
\hatt u_b(\lambda_0,x), & \text{for $x$ near b},  \end{cases}   \quad
v_2(x) = \begin{cases} u_a(\lambda_0,x), & \text{for $x$ near a}, \\
u_b(\lambda_0,x), & \text{for $x$ near b},  \end{cases}   \lb{7.3.33A}
\end{align}
for each $g \in \dom(T_{max})$, the following limits exist and are finite:
\begin{align}
\begin{split}
\wti g(a) &:= - W(v_2, g)(a) =  - W(u_a(\lambda_0,\dott), g)(a)
= \lim_{x \downarrow a} \f{g(x)}{\hatt u_a(\lambda_0,x)},    \\
\wti g(b) &:= - W(v_2, g)(b) =  - W(u_b(\lambda_0,\dott), g)(b) = \lim_{x \uparrow b} \f{g(x)}{\hatt u_b(\lambda_0,x)},
\end{split} \lb{7.3.34A} \\
\begin{split}
{\wti g}^{\, \prime}(a) &:= W(v_1, g)(a) = W(\hatt u_a(\lambda_0,\dott), g)(a)
= \lim_{x \downarrow a} \f{g(x) - \wti g(a) \hatt u_a(\lambda_0,x)}{u_a(\lambda_0,x)},    \\
{\wti g}^{\, \prime}(b) &:= W(v_1, g)(b) = W(\hatt u_b(\lambda_0,\dott), g)(b)
= \lim_{x \uparrow b} \f{g(x) - \wti g(b) \hatt u_b(\lambda_0,x)}{u_b(\lambda_0,x)}.   \lb{7.3.34B}
\end{split}
\end{align}
\end{proposition}

\begin{definition}
The quantities $\wti g(c)$, $\wti g^{\, \prime}(c)$, $c\in \{a,b\}$, defined by \eqref{7.3.34A} and \eqref{7.3.34B} are called the \textit{generalized boundary values} of $g \in \dom(T_{max})$.
\end{definition}

If $\tau$ is in the limit circle case at both endpoints of $(a,b)$, then $T_{min}$ has deficiency indices $(2,2)$.  In this case, the self-adjoint extensions of $T_{min}$ are parametrized by boundary conditions at the endpoints of $(a,b)$ according to the next proposition.

\begin{proposition} \lb{23.t7.t3.12}
Assume Hypothesis \ref{h2.1} and let $\tau$ be in the limit circle case at $a$ and $b$.  In addition, assume that $T_{min} \geq \lambda_0 I_{\Lr}$ for some $\lambda_0 \in \bbR$ and that $u_t(\lambda_0,\dott)$ and $\widehat u_t(\lambda_0,\dott)$ are principal and nonprincipal solutions of $\tau u = \lambda_0 u$ on $(a,b)$, respectively, at $t\in \{a,b\}$ that satisfy \eqref{7.3.33AB}. Then, given \eqref{7.3.34A} and \eqref{7.3.34B}, the following items $(i)$--\,$(v)$ hold\,$:$ \\[1mm]
$(i)$ The minimal operator is characterized by
\begin{align}
\begin{split}
& T_{min} f = \tau f, \\
& f \in \dom(T_{min})= \big\{g\in\dom(T_{max})  \, \big| \, \wti g(a) = {\wti g}^{\, \prime}(a) =0
= \wti g(b) = {\wti g}^{\, \prime}(b)\big\}.      \lb{23.13.2.27b}
\end{split}
\end{align}
$(ii)$ All self-adjoint extensions $T_{\al,\be}$ of $T_{min}$ with separated boundary conditions are of the form
\begin{align}
\begin{split}
& T_{\al,\be} f = \tau f, \quad \al,\be \in[0,\pi),      \lb{23.13.2.27} \\
& f \in \dom(T_{\al,\be})=\left\{g\in\dom(T_{max}) \, \Bigg| \,
\begin{matrix} \sin(\al) {\wti g}^{\, \prime}(a)
+ \cos(\al) \wti g(a) = 0; \\ \sin(\be) {\wti g}^{\, \prime}(b) + \cos(\be) \wti g(b) = 0 \end{matrix} \right\}.
\end{split}
\end{align}
$(iii)$ All self-adjoint extensions $T_{\varphi,R}$ of $T_{min}$ with coupled boundary conditions are of the form
\begin{align}
\begin{split}
& T_{\varphi,R} f = \tau f,\quad \varphi \in [0,\pi),\, R \in SL(2,\bbR),  \\
& f \in \dom(T_{\varphi,R})=\left\{g\in\dom(T_{max}) \, \Bigg| \begin{pmatrix} \wti g(b)
\\ {\wti g}^{\, \prime}(b) \end{pmatrix} = e^{i\varphi}R \begin{pmatrix}
\wti g(a) \\ {\wti g}^{\, \prime}(a) \end{pmatrix} \right\}. \lb{23.13.2.27a}
\end{split}
\end{align}
$(iv)$ Every self-adjoint extension of $T_{min}$ is either of type $(ii)$ $($i.e., with separated boundary conditions\,$)$ or of type $(iii)$
$($i.e., with coupled boundary conditions\,$)$.\\[1mm]
$(v)$ The operator $T_{\alpha=0,\beta=0}$ is the Friedrichs extension of $T_{min}$.
\end{proposition}

In the case when exactly one endpoint is in the limit circle case, the deficiency indices of $T_{min}$ are $(1,1)$.  The self-adjoint extensions of $T_{min}$ are then characterized by a separated boundary condition at the limit circle endpoint.  For simplicity of presentation, we assume in the following result that $\tau$ is in the limit circle case at $a$ (the case when $\tau$ is in the limit circle case at $b$ is entirely analogous).

\begin{proposition}\lb{p4.1} 
Assume Hypothesis \ref{h2.1} and let $\tau$ be in the limit circle case at $a$ and in the limit point case at $b$. In addition, assume that $T_{min} \geq \lambda_0 I_{\Lr}$ for some $\lambda_0 \in \bbR$ and that $u_a(\lambda_0,\dott)$ and $\widehat u_a(\lambda_0,\dott)$ are principal and nonprincipal solutions of $\tau u = \lambda_0 u$ on $(a,b)$, respectively, at $a$ that satisfy \eqref{7.3.33AB}.  Introduce the corresponding generalized boundary values according to \eqref{7.3.34A} and \eqref{7.3.34B}.  Then the following statements $(i)$--\,$(iii)$ hold:\\[1mm]
$(i)$  The domain of $T_{min}$ is characterized by
\begin{equation}
\dom(T_{min})=\{g\in \dom(T_{max})\,|\, \wti g^{\,\prime}(a)=\wti g(a)=0\}.
\end{equation}
$(ii)$ An operator $T$ in $\Lr$ is a self-adjoint extension of $T_{min}$ if and only if $T=T_{\al}$, for some $\alpha\in [0,\pi)$, where
\begin{align}
\begin{split}
& T_{\al} f = \tau f, \quad \al \in[0,\pi),     \\
& f \in \dom(T_{\al})=\big\{g\in\dom(T_{max}) \, \big| \, \sin(\al) {\wti g}^{\, \prime}(a)
+ \cos(\al) \wti g(a) = 0\big\}.
\end{split}
\end{align}
$(iii)$ The operator $T_{\alpha=0}$ is the Friedrichs extension of $T_{min}$.\\[1mm]
Results analogous to $(i)$--\,$(iii)$ hold if $\tau$ is in the limit point case at $x=a$ and in the limit circle case at $x=b$.
\end{proposition}

In the case when $\tau$ is in the limit point case at both $a$ and $b$, the deficiency indices of $T_{min}$ are $(0,0)$.  In this case, $T:=T_{min}=T_{max}$ is self-adjoint.

\begin{proposition}\lb{p2.14}
Assume Hypothesis \ref{h2.1}.  If $\tau$ is in the limit point case at both $a$ and $b$, then $T:=T_{min}=T_{max}$ is self-adjoint.
\end{proposition}

\begin{remark} \lb{r2.15} 
The generalized boundary values associated with the Sturm--Liouville expression
(2.2) as introduced in Proposition \ref{23.t7.3.11} by
 \begin{align}
\wti g(c) &= \lim_{x \to c} \f{g(x)}{\hatt u_c(\lambda_0,x)},   \lb{7.3.43} \\
{\wti g}^{\, \prime}(c) & = \lim_{x \to c} \f{g(x) - \wti g(c) \hatt u_c(\lambda_0,x)}{u_c(\lambda_0,x)},     \lb{7.3.44}
\end{align}
especially, $\wti g(c)$ in \eqref{7.3.43}, at an endpoint $c \in \{a,b\}$, have a longer history.
 
$(i)$ Rellich \cite{Re43} in connection with coefficients $p, q, r$ that had a very particular behavior in a neighborhood of the endpoint $c$ of the type
\index{generalized boundary values}
\begin{align}
\begin{split}
p(x) &= (x - c)^{\sigma} \big[p_0 + p_1 (x - c) + p_2 (x - c)^2 + \cdots \big],  \\
q(x) &= (x - c)^{\sigma - 2} \big[q_0 + q_1 (x - c) + q_2 (x - c)^2 + \cdots \big],  \\
r(x) &= (x - c)^{\sigma - 2} \big[r_0 + r_1 (x - c) + r_2 (x - c)^2 + \cdots \big],  \\
\end{split}
\end{align}
with $\sigma, p_0, p_1, \dots , q_0, q_1, \dots , r_0, r_1, \dots  \in \bbR$, $p_0 \neq 0$, $r_k \neq 0$ for some
$k \in \bbN_0$, $k_{\ell} = 0$ for $0 \leq \ell \leq k-1$, etc. This was also recorded in \cite[Ch.~15]{He67} and \cite[Ch.~III]{JR76}.
In 1951, Rellich \cite{Re51} considerably generalized the hypotheses on $p,q,r$. The case of the Bessel equation was reconsidered in
\cite{GP84}, and the case of Schr\"odinger operators on $(0,\infty)$ with potentials $q$ satisfying
\begin{equation}
q(x) = \big(\gamma^2 - (1/4)\big) x^{-2} + \eta x^{-1} + \omega x^{-a} + W(x) \,
\text{ for a.e.~$x > 0$},
\end{equation}
with $\gamma \geq 0$, $\eta, \omega \in \bbR$, $a \in (0,2)$, and $W \in L^{\infty}((0,\infty);dx)$ real-valued a.e.,  was systematically treated in \cite{BG85} and \cite{Ko91}. Under the general Hypothesis \ref{h2.1}, the boundary value $\wti g(c)$ in \eqref{7.3.43} was studied in detail by Kalf \cite[Remark~3]{Ka78} and subsequently by Rosenberger in \cite[Theorem~3]{Ro85}. It was once again systematically employed by Niessen and Zettl \cite{NZ92}. In this context we also refer to \cite[Propositions~6.11.1, 6.12.1]{BHS20},
which discusses linearly independent boundary values in terms of boundary triplets and Wronskians
$W(\hatt u_b(\lambda_0,\dott), g)(c)$. \\[1mm]
$(ii)$ The difference quotient analogue of  ${\wti g}^{\, \prime}(c)$ in \eqref{7.3.44}, on the other hand, apparently, was not considered in \cite{BHS20}, \cite{Ka78}, \cite{NZ92}, and \cite{Ro85}. It is a new twist in \cite{GLN20} that offers an explicit description of boundary conditions for lower semibounded, self-adjoint, singular (quasi-regular) Sturm--Liouville operators. \\[1mm] 
$(iii)$ 
We recall that for an element $g \in \dom (T_{max})$ 
the conditions $\widetilde g(a)=\widetilde g(b)=0$ describe the Friedrichs extension 
in Proposition \ref{23.t7.t3.12}, and the condition $\widetilde g(a)=0$
describes the Friedrichs extension  in Proposition \ref{p4.1}.
It is worthwhile to observe that for $c \in \{a,b\}$ a condition of the form 
$\widetilde g(c)=0$ is sometimes met in a different guise, such as
\begin{equation}
 \lim_{x \to c} \frac{g(x)}{u_c(\lambda_0,x)} \mbox{ exists in } \bbC,
\end{equation} 
where $u_c(\lambda_0, \cdot)$ is a principal solution. For the special case of the Legendre operator see, for instance,  \cite[Sect.~132]{AG81}. 
For the above and other alternative statements, see also
\cite[Corollary 6.11.9, Corollary 6.12.9]{BHS20} and \cite[Sect.~13.4]{GNZ24}. 
~\hfill $\diamond$
\end{remark}

\section{Case One:  Two Limit Circle Endpoints} \lb{s3}

In this section we investigate the situation when $\tau$ is in the limit circle case at both $a$ and $b$.
The main goal is to provide the sesquilinear forms corresponding to the lower semibounded self-adjoint extensions of $T_{min}$ with separated and coupled boundary conditions from Proposition \ref{23.t7.t3.12}.
The following hypothesis is assumed throughout this section.

\begin{hypothesis}\lb{h3.1}
In addition to Hypothesis \ref{h2.1}, assume that $\tau$ is in the limit circle case at $a$ and $b$.  Suppose that $T_{min} \geq \lambda_0 I_{\Lr}$ for some $\lambda_0 \in \bbR$ and that $u_t(\lambda_0,\dott)$ and $\widehat u_t(\lambda_0,\dott)$ are principal and nonprincipal solutions of $\tau u = \lambda_0 u$ on $(a,b)$, respectively, at $t\in \{a,b\}$ that satisfy \eqref{7.3.33AB}.
\end{hypothesis}

Assuming Hypothesis \ref{h3.1}, choose $a_0,b_0\in (a,b)$ such that $a<a_0<b_0<b$ and
\begin{equation}\lb{3.1}
\begin{split}
&u_a(\lambda_0,x)\neq 0,\quad \widehat{u}_a(\lambda_0,x)\neq 0,\quad x\in (a,a_0);\\
&u_b(\lambda_0,x)\neq 0,\quad \widehat{u}_b(\lambda_0,x)\neq 0,\quad x\in (b_0,b).
\end{split}
\end{equation}
Let $c\in (a,a_0)$ and $d\in (b_0,b)$ be fixed.  Introducing the differential expressions $N_{\widehat{u}_a(\lambda_0,\dott),c}$ and $N_{\widehat{u}_b(\lambda_0,\dott),d}$ by
\begin{align}
N_{\widehat{u}_a(\lambda_0,\dott),c}f &= p^{1/2}\widehat{u}_a(\lambda_0,\dott)\bigg(\frac{f}{\widehat{u}_a(\lambda_0,\dott)}\bigg)',\quad f\in AC_{loc}((a,c));\lb{3.2vv}\\
N_{\widehat{u}_b(\lambda_0,\dott),d}g &= p^{1/2}\widehat{u}_b(\lambda_0,\dott)\bigg(\frac{g}{\widehat{u}_b(\lambda_0,\dott)}\bigg)',\quad g\in AC_{loc}((d,b)),
\end{align}
one defines the symmetric sesquilinear form $\mathfrak{Q}_{c,d}$ as follows, see, for instance, \cite[Sect.~6.8]{BHS20}, \cite[Sect.~4.5]{GNZ24},
\begin{align}
&\dom(\mathfrak{Q}_{c,d})=\big\{h\in \Lr\,\big|\, h\in AC_{loc}((a,b)),\no\\
&\hspace*{2.3cm} p^{-1/2}h^{[1]}\in L^2((c,d);dx),\, N_{\widehat{u}_a(\lambda_0,\dott),c}h\in L^2((a,c);dx),\lb{3.4}\\
&\hspace*{6.4cm} N_{\widehat{u}_b(\lambda_0,\dott),d}h\in L^2((d,b);dx)\big\},\no
\end{align}
and
\begin{align}
\mathfrak{Q}_{c,d}(f,g) &= \int_a^c dx\, \ol{(N_{\widehat{u}_a(\lambda_0,\dott),c}f)(x)} (N_{\widehat{u}_a(\lambda_0,\dott),c}g)(x)  \no \\
& \quad + \int_d^b dx\, \ol{(N_{\widehat{u}_b(\lambda_0,\dott),d}f)(x)} (N_{\widehat{u}_b(\lambda_0,\dott),d}g)(x)    \no\\
&\quad + \lambda_0 \int_a^c r(x) \,dx\, \ol{f(x)}g(x) + \lambda_0\int_d^b r(x)\,dx\, \ol{f(x)}g(x)    \no \\
& \quad + \int_c^d\, dx\, \Big[p(x)^{-1}\ol{f^{[1]}(x)}g^{[1]}(x) + q(x) \ol{f(x)}g(x)\Big]\no\\
&\quad+ \frac{\widehat{u}_a^{[1]}(\lambda_0,c)}{\widehat{u}_a(\lambda_0,c)}\ol{f(c)}g(c) - \frac{\widehat{u}_b^{[1]}(\lambda_0,d)}{\widehat{u}_b(\lambda_0,d)}\ol{f(d)}g(d),\quad f,g\in \dom(\mathfrak{Q}_{c,d}).\lb{3.5}
\end{align}

Several important properties of the sesquilinear form $\mathfrak{Q}_{c,d}$ are collected in the following result.

\begin{proposition}\lb{p3.1}
Assume Hypothesis \ref{h3.1}.  Let $a<a_0<b_0<b$ with $a_0$ and $b_0$ chosen so that \eqref{3.1} holds and suppose $c\in (a,a_0)$ and $d\in (b_0,b)$.  Then the following statements $(i)$--\,$(iv)$ hold:\\[1mm]
$(i)$ The sesquilinear form $\mathfrak{Q}_{c,d}$ defined by \eqref{3.4} and \eqref{3.5} is densely defined, closed, and lower semibounded in $\Lr$.\\[1mm]
$(ii)$ $\dom(T_{max})\subseteq \dom(\mathfrak{Q}_{c,d})$.\\[1mm]
$(iii)$  If $c'\in (a,a_0)$ and $d'\in (b_0,b)$, then $\mathfrak{Q}_{c,d}=\mathfrak{Q}_{c',d'}$.  That is, the sesquilinear form defined by \eqref{3.4} and \eqref{3.5} is independent of the choices of $c\in (a,a_0)$ and $d\in (b_0,b)$.\\[1mm]
$(iv)$  If $g\in \dom(\mathfrak{Q}_{c,d})$, then the following limits exist:
\begin{equation}\lb{3.6}
\wti g(a) := \lim_{x \downarrow a} \f{g(x)}{\hatt u_a(\lambda_0,x)},\quad \wti g(b) := \lim_{x \uparrow b} \f{g(x)}{\hatt u_b(\lambda_0,x)}.
\end{equation}
In particular, the generalized boundary values $\wti g(a)$ and $\wti g(b)$ introduced in \eqref{7.3.34A} for functions in $\dom(T_{max})$ extend to functions in $\dom(\mathfrak{Q}_{c,d})$.
\end{proposition}

\begin{remark}
The properties of $\mathfrak{Q}_{c,d}$ in Proposition \ref{p3.1} are discussed in detail in \cite{{BHS20}}; see \cite[Theorem 6.10.9, Lemma 6.9.4, Corollary 6.11.2, Lemma 6.11.3]{BHS20}. \hfill$\diamond$
\end{remark}

\begin{lemma}\lb{l3.3c}
Assume Hypothesis \ref{h3.1}.  Let $a<a_0<b_0<b$ with $a_0$ and $b_0$ chosen so that \eqref{3.1} holds and suppose $c\in (a,a_0)$ and $d\in (b_0,b)$.  If $f\in \dom(\mathfrak{Q}_{c,d})$ and $g\in \dom(T_{max})$, then
\begin{align}
(f,T_{max}g)_{\Lr} = \mathfrak{Q}_{c,d}(f,g) + \overline{\wti f(a)}\wti g^{\,\prime}(a) - \overline{\wti f(b)}\wti g^{\,\prime}(b).\lb{3.7c}
\end{align}
\end{lemma}
\begin{proof}
We recall Jacobi's factorization identity in the following form: If $g,h\in AC_{loc}((a,b))$ and $g^{[1]},h^{[1]}\in AC_{loc}((a,b))$, then
\begin{equation}
    -\big(g^{[1]}\big)' + \frac{\big(h^{[1]}\big)'}{h}g = -\frac{1}{h}\Bigg[ph^2 \bigg(\frac{g}{h}\bigg)'\Bigg]'\, \text{ when $h\neq 0$}.
\end{equation}
Since $\widehat{u}_t(\lambda_0,\dott)$, $t\in\{a,b\}$, are solutions of $\tau u = \lambda_0 u$ on $(a,b)$, one infers that:
\begin{align}
    q &= \lambda_0 r + \frac{\big(\widehat{u}_a^{[1]}(\lambda_0,\dott)\big)'}{\widehat{u}_a(\lambda_0,\dott)}\, \text{ a.e.~on $(a,a_0)$};\lb{3.9m}\\
    q &= \lambda_0 r + \frac{\big(\widehat{u}_b^{[1]}(\lambda_0,\dott)\big)'}{\widehat{u}_b(\lambda_0,\dott)}\, \text{ a.e.~on $(b_0,b)$}.
\end{align}
To prove \eqref{3.7c} one calculates for $f\in \dom(\mathfrak{Q}_{c,d})$ and $g\in \dom(T_{max})$ as follows:
\begin{align}
&(f,T_{max}g)_{\Lr}\no\\
&\quad = \int_a^b dx\, \ol{f}\big[-\big(g^{[1]}\big)'+qg\big]\no\\
&\quad = \lim_{a'\downarrow a}\int_{a'}^cdx\, \ol{f}\bigg[-\big(g^{[1]}\big)'+\lambda_0rg + \frac{\big(\widehat{u}_a^{[1]}(\lambda_0,\dott)\big)'}{\widehat{u}_a(\lambda_0,\dott)}g \bigg] - \int_c^ddx\, \ol{f}\big(g^{[1]}\big)'\no\\
&\qquad + \int_c^ddx\, q\ol{f}g + \lim_{b'\uparrow b}\int_d^{b'}dx\, \ol{f}\bigg[-\big(g^{[1]}\big)'+\lambda_0rg + \frac{\big(\widehat{u}_b^{[1]}(\lambda_0,\dott)\big)'}{\widehat{u}_b(\lambda_0,\dott)}g\bigg]\no\\
&\quad = \lim_{a'\downarrow a}\int_{a'}^cdx\, \ol{f}\Bigg\{-\frac{1}{\widehat{u}_a(\lambda_0,\dott)}\Bigg[p\widehat{u}_a(\lambda_0,\dott)^2\bigg(\frac{g}{\widehat{u}_a(\lambda_0,\dott)}\bigg)'\Bigg]' \Bigg\}\no\\
&\qquad + \lambda_0 \int_a^c r\,dx\, \overline{f}g - \overline{f}g^{[1]}\big|_c^d + \int_c^d dx\, \big(p^{-1}\overline{f^{[1]}}g^{[1]} + q\overline{f}g\big) + \lambda_0 \int_d^b r\,dx\, \overline{f}g\no\\
&\qquad + \lim_{b'\uparrow b}\int_d^{b'}dx\, \ol{f}\Bigg\{-\frac{1}{\widehat{u}_b(\lambda_0,\dott)}\Bigg[p\widehat{u}_b(\lambda_0,\dott)^2\bigg(\frac{g}{\widehat{u}_b(\lambda_0,\dott)}\bigg)'\Bigg]' \Bigg\}\no\\
&\quad = \lim_{a'\downarrow a}\Bigg\{-\frac{\ol{f}}{\widehat{u}_a(\lambda_0,\dott)}\Bigg[p\widehat{u}_a(\lambda_0,\dott)^2\bigg(\frac{g}{\widehat{u}_a(\lambda_0,\dott)}\bigg)'\Bigg]\Bigg|_{a'}^c\no\\
&\hspace*{1.8cm} + \int_{a'}^cdx\, \bigg(\frac{\overline{f}}{\widehat{u}_a(\lambda_0,\dott)}\bigg)' p \widehat{u}_a(\lambda_0,\dott)^2\bigg(\frac{g}{\widehat{u}_a(\lambda_0,\dott)}\bigg)' \Bigg\}\no\\
&\qquad + \lambda_0 \int_a^c r\,dx\, \overline{f}g - \overline{f}g^{[1]}\big|_c^d + \int_c^d dx\, \big(p^{-1}\overline{f^{[1]}}g^{[1]} + q\overline{f}g\big) + \lambda_0 \int_d^b r\,dx\, \overline{f}g\no\\
&\qquad + \lim_{b'\uparrow b}\Bigg\{-\frac{\ol{f}}{\widehat{u}_b(\lambda_0,\dott)}\Bigg[p\widehat{u}_b(\lambda_0,\dott)^2\bigg(\frac{g}{\widehat{u}_b(\lambda_0,\dott)}\bigg)'\Bigg]\Bigg|_d^{b'}\no\\
&\hspace{2cm} + \int_d^{b'}dx\, \bigg(\frac{\overline{f}}{\widehat{u}_b(\lambda_0,\dott)}\bigg)' p \widehat{u}_b(\lambda_0,\dott)^2\bigg(\frac{g}{\widehat{u}_b(\lambda_0,\dott)}\bigg)' \Bigg\}.\lb{3.13c}
\end{align}
The evaluation terms at $c$ and $d$ in \eqref{3.13c} are
\begin{align}
&\Bigg\{-\frac{\ol{f}}{\widehat{u}_a(\lambda_0,\dott)}\Bigg[p\widehat{u}_a(\lambda_0,\dott)^2\bigg(\frac{g}{\widehat{u}_a(\lambda_0,\dott)}\bigg)'\Bigg]\Bigg\}(c)\no\\
&\qquad - \Bigg\{-\frac{\ol{f}}{\widehat{u}_b(\lambda_0,\dott)}\Bigg[p\widehat{u}_b(\lambda_0,\dott)^2\bigg(\frac{g}{\widehat{u}_b(\lambda_0,\dott)}\bigg)'\Bigg]\Bigg\}(d) - \overline{f(d)}g^{[1]}(d) + \overline{f(c)}g^{[1]}(c)\no\\
&\quad=-\frac{\overline{f(c)}}{\widehat{u}_a(\lambda_0,c)}\big\{g^{[1]}(c)\widehat{u}_a(\lambda_0,c)-g(c)\widehat{u}_a^{[1]}(\lambda_0,c)\big\} - \ol{f(d)}g^{[1]}(d) + \ol{f(c)}g^{[1]}(c)\no\\
&\qquad + \frac{\ol{f(d)}}{\widehat{u}_b(\lambda_0,d)}\big\{g^{[1]}(d)\widehat{u}_b(\lambda_0,d) - g(d)\widehat{u}_b^{[1]}(\lambda_0,d)\big\}\no\\
&\quad = \frac{\widehat{u}_a^{[1]}(\lambda_0,c)}{\widehat{u}_a(\lambda_0,c)}\ol{f(c)}g(c) - \frac{\widehat{u}_b^{[1]}(\lambda_0,d)}{\widehat{u}_b(\lambda_0,d)}\ol{f(d)}g(d).\lb{3.12c}
\end{align}

Applying \eqref{7.3.34B} and \eqref{3.6}, one obtains for $f\in \dom(\mathfrak{Q}_{c,d})$ and $g\in \dom(T_{max})$,
\begin{align}
&\lim_{a'\downarrow a} \Bigg[\frac{\overline{f}}{\widehat{u}_a(\lambda_0,\dott)}p\widehat{u}_a(\lambda_0,\dott)^2\bigg(\frac{g}{\widehat{u}_a(\lambda_0,\dott)}\bigg)'\Bigg](a')\no\\
&\quad = \lim_{a'\downarrow a} \frac{\overline{f(a')}}{\widehat{u}_a(\lambda_0,a')}W(\widehat{u}_a(\lambda_0,\dott),g)(a') = \overline{\wti f(a)}\wti g^{\,\prime}(a),\lb{3.14c}
\end{align}
and, similarly,
\begin{equation}
\lim_{b'\uparrow b} \Bigg[-\frac{\overline{f}}{\widehat{u}_b(\lambda_0,\dott)}p\widehat{u}_b(\lambda_0,\dott)^2\bigg(\frac{g}{\widehat{u}_b(\lambda_0,\dott)}\bigg)'\Bigg](b') = -\overline{\wti f(b)}\wti g^{\,\prime}(b).\lb{3.15c}
\end{equation}
In light of \eqref{3.12c}, \eqref{3.14c}, and \eqref{3.15c}, \eqref{3.13c} reduces to \eqref{3.7c}.
\end{proof}

\begin{remark}
The identity \eqref{3.7c} may be found written in the language of boundary triplets in \cite[Equation (6.11.5)]{BHS20}; see also Appendix \ref{{sA}}.\hfill$\diamond$
\end{remark}

The following infinitesimal form boundedness result is a consequence of \cite[Lemma 6.10.4]{BHS20}.

\begin{proposition}\lb{p3.3}
Assume Hypothesis \ref{h3.1}.  For every $\varepsilon>0$ there exists $C(\varepsilon)>0$ such that
\begin{equation}\lb{3.7a}
\big|\wti f(t)\big|^2 \leq \varepsilon\mathfrak{Q}_{c,d}(f,f) + C(\varepsilon)\|f\|_{\Lr}^2,\quad f\in \dom(\mathfrak{Q}_{c,d}),\, t\in\{a,b\}.
\end{equation}
\end{proposition}

\begin{remark}\lb{r3.5c}
It is clear that the inequality in \eqref{3.7a} remains valid with $\mathfrak{Q}_{c,d}(f,f)$ replaced by $|\mathfrak{Q}_{c,d}(f,f)|$, $f\in \dom(\mathfrak{Q}_{c,d})$.  In particular, the sesquilinear forms $\mathfrak{q}_t(f,g):=\ol{\wti f(t)}\wti g(t)$, $f,g\in \dom(\mathfrak{q}_t):=\dom(\mathfrak{Q}_{c,d})$, $t\in \{a,b\}$, are infinitesimally bounded with respect to $\mathfrak{Q}_{c,d}$.\hfill$\diamond$
\end{remark}

In the next theorem we provide the sesquilinear form corresponding to the self-adjoint extensions $T_{\alpha,\beta}$, $\alpha,\beta\in[0,\pi)$, of $T_{min}$ with separated boundary conditions from Proposition \ref{23.t7.t3.12}\,$(ii)$.

\begin{theorem}\label{case0}
Assume Hypothesis \ref{h3.1}.  Let $a<a_0<b_0<b$ with $a_0$ and $b_0$ chosen so that \eqref{3.1} holds and suppose $c\in (a,a_0)$ and $d\in (b_0,b)$.  Then the following statements $(i)$--\,$(iv)$ hold:\\[1mm]
$(i)$ If $\alpha,\beta\in (0,\pi)$, then the sesquilinear form $\mathfrak{Q}_{c,d}^{\alpha,\beta}$ defined by
\begin{equation}\lb{3.8}
\begin{split}
\mathfrak{Q}_{c,d}^{\alpha,\beta}(f,g)=\mathfrak{Q}_{c,d}(f,g) + \cot(\beta)\ol{\wti f(b)}\wti g(b) - \cot(\alpha)\ol{\wti f(a)}\wti g(a),&\\
f,g\in \dom\big(\mathfrak{Q}_{c,d}^{\alpha,\beta}\big) = \dom(\mathfrak{Q}_{c,d}),&
\end{split}
\end{equation}
is densely defined, closed, symmetric, and lower semibounded.  In addition,
\begin{equation}\lb{3.9a}
(f,T_{\alpha,\beta}g)_{\Lr} = \mathfrak{Q}_{c,d}^{\alpha,\beta}(f,g),\quad f\in \dom\big(\mathfrak{Q}_{c,d}^{\alpha,\beta}\big),\, g\in \dom(T_{\alpha,\beta}).
\end{equation}
Hence, $\mathfrak{Q}_{c,d}^{\alpha,\beta}$ is the unique densely defined, closed, symmetric, lower semibounded sesquilinear form uniquely associated to $T_{\alpha,\beta}$, $\alpha,\beta\in (0,\pi)$, by the First Representation Theorem $($cf.~\cite[Theorem VI.2.1]{Ka80}$)$.\\[1mm]
$(ii)$ If $\alpha=0$ and $\beta\in (0,\pi)$, then the sesquilinear form defined by
\begin{equation}\lb{3.10a}
\begin{split}
&\mathfrak{Q}_{c,d}^{0,\beta}(f,g) = \mathfrak{Q}_{c,d}(f,g) + \cot(\beta)\ol{\wti f(b)}\wti g(b),\\
&f,g\in \dom\big(\mathfrak{Q}_{c,d}^{0,\beta}\big)=\big\{h\in \dom(\mathfrak{Q}_{c,d})\,\big|\, \wti h(a)=0\big\},
\end{split}
\end{equation}
is densely defined, closed, symmetric, and lower semibounded.  In addition,
\begin{equation}\lb{3.11a}
(f,T_{0,\beta}g)_{\Lr} = \mathfrak{Q}_{c,d}^{0,\beta}(f,g),\quad f\in \dom\big(\mathfrak{Q}_{c,d}^{0,\beta}\big),\, g\in \dom(T_{0,\beta}).
\end{equation}
Hence, $\mathfrak{Q}_{c,d}^{0,\beta}$ is the unique densely defined, closed, symmetric, lower semibounded sesquilinear form uniquely associated to $T_{0,\beta}$, $\beta\in (0,\pi)$, by the First Representation Theorem.\\[1mm]
$(iii)$ If $\alpha\in(0,\pi)$ and $\beta=0$, then the sesquilinear form defined by
\begin{equation}\lb{3.20aa}
\begin{split}
&\mathfrak{Q}_{c,d}^{\alpha,0}(f,g) = \mathfrak{Q}_{c,d}(f,g) - \cot(\alpha)\ol{\wti f(a)}\wti g(a),\\
&f,g\in \dom\big(\mathfrak{Q}_{c,d}^{\alpha,0}\big)=\big\{h\in \dom(\mathfrak{Q}_{c,d})\,\big|\, \wti h(b)=0\big\},
\end{split}
\end{equation}
is densely defined, closed, lower semibounded, symmetric.  In addition,
\begin{equation}
(f,T_{\alpha,0}g)_{\Lr} = \mathfrak{Q}_{c,d}^{\alpha,0}(f,g),\quad f\in \dom\big(\mathfrak{Q}_{c,d}^{\alpha,0}\big),\, g\in \dom(T_{\alpha,0}).
\end{equation}
Hence, $\mathfrak{Q}_{c,d}^{\alpha,0}$ is the unique densely defined, closed, symmetric, lower semibounded sesquilinear form uniquely associated to $T_{\alpha,0}$, $\alpha\in (0,\pi)$, by the First Representation Theorem.\\[1mm]
$(iv)$ If $\alpha=\beta=0$, then the sesquilinear form defined by
\begin{equation}\lb{3.22aa}
\begin{split}
&\mathfrak{Q}_{c,d}^{0,0}(f,g) = \mathfrak{Q}_{c,d}(f,g),\\
&f,g\in \dom\big(\mathfrak{Q}_{c,d}^{0,0}\big)=\big\{h\in \dom(\mathfrak{Q}_{c,d})\,\big|\, \wti h(a)=0=\wti h(b)\big\},
\end{split}
\end{equation}
is densely defined, closed, symmetric, and lower semibounded.  In addition,
\begin{equation}
(f,T_{0,0}g)_{\Lr} = \mathfrak{Q}_{c,d}^{0,0}(f,g),\quad f\in \dom\big(\mathfrak{Q}_{c,d}^{0,0}\big),\, g\in \dom(T_{0,0}).
\end{equation}
Hence, $\mathfrak{Q}_{c,d}^{0,0}$ is the unique densely defined, closed, symmetric, lower semibounded sesquilinear form uniquely associated to $T_{0,0}$ by the First Representation Theorem.  
\end{theorem}
\begin{proof}
It is clear by inspection that $\mathfrak{Q}_{c,d}^{\alpha,\beta}$ is symmetric.  That $\mathfrak{Q}_{c,d}^{\alpha,\beta}$ is densely defined, closed, and lower semibounded follows from Remark \ref{r3.5c} (specifically, the infinitesimal form boundedness of $\mathfrak{q}_t$, $t\in\{a,b\}$, with respect to $\mathfrak{Q}_{c,d}$).  To establish \eqref{3.8}, one applies Lemma \ref{l3.3c}---specifically \eqref{3.7c}---and the boundary conditions inherent in the definition of $\dom(T_{\alpha,\beta})$:
\begin{equation}
\begin{split}
(f,T_{\alpha,\beta}g)_{\Lr} &= \mathfrak{Q}_{c,d}(f,g) + \overline{\wti f(a)}\wti g^{\,\prime}(a) - \overline{\wti f(b)}\wti g^{\,\prime}(b)\\
&= \mathfrak{Q}_{c,d}(f,g) - \cot(\alpha)\overline{\wti f(a)}\wti g(a) + \cot(\beta)\overline{\wti f(b)}\wti g(b)\\
&= \mathfrak{Q}_{c,d}^{\alpha,\beta}(f,g),\quad f\in \dom\big(\mathfrak{Q}_{c,d}^{\alpha,\beta}\big),\, g\in \dom(T_{\alpha,\beta}).
\end{split}
\end{equation}

The proofs of $(ii)$, $(iii)$, and $(iv)$ are all similar.  We will provide a sketch of the proof of the claims in $(ii)$ and omit the details for $(iii)$ and $(iv)$.  One notes that $\mathfrak{Q}_{c,d}^{0,\beta}$ is densely defined since $\dom(T_{min})\subseteq \dom\big(\mathfrak{Q}_{c,d}^{0,\beta}\big)$ and $T_{min}$ is densely defined.  Moreover, $\mathfrak{Q}_{c,d}^{0,\beta}$ is lower semibounded since it is a restriction of $\mathfrak{Q}_{c,d}^{\pi/2,\beta}$, and the latter is lower semibounded by part $(i)$.

To prove item $(ii)$, one notes that $\dom(T_{min})\subseteq \dom\big(\mathfrak{Q}_{c,d}^{0,\beta}\big)$, so $\mathfrak{Q}_{c,d}^{0,\beta}$ is densely defined.  Let $\mathfrak{Q}_{c,d}'$ denote the restriction of $\mathfrak{Q}_{c,d}$ to $\dom\big(\mathfrak{Q}_{c,d}^{0,\beta}\big)$, where the latter domain is defined according to \eqref{3.10a}.  Since $\mathfrak{Q}_{c,d}^{0,\beta}$ is an infinitesimally form bounded perturbation of $\mathfrak{Q}_{c,d}'$ by \eqref{3.7a}, to prove $\mathfrak{Q}_{c,d}^{0,\beta}$ is closed, it suffices to show that $\mathfrak{Q}_{c,d}'$ is closed.  If $\{f_n\}_{n=1}^{\infty}\subset \dom(\mathfrak{Q}_{c,d}')=\dom\big(\mathfrak{Q}_{c,d}^{0,\beta}\big)$, $\|f_n-f\|_{\Lr}\to 0$ for some $f\in \Lr$, and $\mathfrak{Q}_{c,d}'(f_n-f_m,f_n-f_m)\to 0$, then the fact that $\mathfrak{Q}_{c,d}$ is closed (cf.~Proposition \ref{p3.1}) implies $f\in \dom(\mathfrak{Q}_{c,d})$ and $\mathfrak{Q}_{c,d}(f_n-f,f_n-f)\to 0$.  Using \eqref{3.7a} one obtains
\begin{align}
\big|\wti f(a)\big|^2 &= \big|\wti f_n(a) - \wti f(a)\big|^2\lb{3.25rr}\\
&\leq \mathfrak{Q}_{c,d}(f_n-f,f_n-f) + C_0\|f_n-f\|_{\Lr}^2,\quad n\in \bbN,\no
\end{align}
for some scalar $C_0\in (0,\infty)$ that does not depend on $n\in \bbN$.  Taking $n\to \infty$ throughout \eqref{3.25rr}, one obtains $\wti f(a) = 0$.  Therefore, $f\in \dom\big(\mathfrak{Q}_{c,d}^{0,\beta}\big)$, and since $\mathfrak{Q}_{c,d}$ is an extension of $\mathfrak{Q}_{c,d}'$, $\mathfrak{Q}_{c,d}'(f_n-f,f_n-f)\to 0$.  Hence, $\mathfrak{Q}_{c,d}'$ is closed, and it follows that $\mathfrak{Q}_{c,d}^{0,\beta}$ is closed and lower semibounded.  That $\mathfrak{Q}_{c,d}^{0,\beta}$ is symmetric is clear by inspection.  Finally, the verification of \eqref{3.11a} is entirely analogous to that of \eqref{3.9a} (invoking Lemma \ref{l3.3c}, etc.), so we omit the details.
\end{proof}

In the next theorem we provide the sesquilinear form corresponding to the self-adjoint extensions $T_{\varphi,R}$, $\varphi\in[0,\pi)$, $R\in SL(2,\bbR)$, of $T_{min}$ with coupled boundary conditions from Proposition \ref{23.t7.t3.12}\,$(iii)$.

\begin{theorem}\label{case1}
Assume Hypothesis \ref{h3.1}.  Let $a<a_0<b_0<b$ with $a_0$ and $b_0$ chosen so that \eqref{3.1} holds and suppose $c\in (a,a_0)$ and $d\in (b_0,b)$.  If $\varphi \in [0,\pi)$ and $R\in SL(2,\bbR)$, then the following statements $(i)$ and $(ii)$ hold:\\[1mm]
$(i)$ If $R_{1,2}\neq 0$, then the sesquilinear form $\mathfrak{Q}_{c,d}^{\varphi,R}$ defined by
\begin{align}
\mathfrak{Q}_{c,d}^{\varphi,R}(f,g) &= \mathfrak{Q}_{c,d}(f,g) - \frac{1}{R_{1,2}}\Big\{R_{1,1}\ol{\wti f(a)}\wti g(a) - e^{-i\varphi}\ol{\wti f(a)}\wti g(b)\lb{3.25}\\
&\quad - e^{i\varphi}\ol{\wti f(b)}\wti g(a) + R_{2,2}\ol{\wti f(b)}\wti g(b) \Big\},\quad f,g\in \dom\big(\mathfrak{Q}_{c,d}^{\varphi,R}\big)=\dom(\mathfrak{Q}_{c,d}),\no
\end{align}
is densely defined, closed, symmetric, and lower semibounded.  In addition,
\begin{equation}\lb{3.26}
(f,T_{\varphi,R}g)_{\Lr} = \mathfrak{Q}_{c,d}^{\varphi,R}(f,g),\quad f\in \dom\big(\mathfrak{Q}_{c,d}^{\varphi,R}\big),\, g\in \dom(T_{\varphi,R}).
\end{equation}
Hence, $\mathfrak{Q}_{c,d}^{\varphi,R}$ is the unique densely defined, closed, symmetric, lower semibounded sesquilinear form uniquely associated to $T_{\varphi,R}$ by the First Representation Theorem.\\[1mm]
$(ii)$ If $R_{1,2}= 0$, then the sesquilinear form $\mathfrak{Q}_{c,d}^{\varphi,R}$ defined by
\begin{align}
&\mathfrak{Q}_{c,d}^{\varphi,R}(f,g) = \mathfrak{Q}_{c,d}(f,g) - R_{1,1}R_{2,1}\ol{\wti f(a)}\wti g(a),\lb{3.27}\\
&f,g\in \dom\big(\mathfrak{Q}_{c,d}^{\varphi,R}\big)=\big\{h\in \dom(\mathfrak{Q}_{c,d})\,\big|\, \wti h(b) = e^{i\varphi}R_{1,1}\wti h(a)\big\},\no
\end{align}
is densely defined, closed, symmetric, and lower semibounded.  In addition,
\begin{equation}\lb{3.28}
(f,T_{\varphi,R}g)_{\Lr} = \mathfrak{Q}_{c,d}^{\varphi,R}(f,g),\quad f\in \dom\big(\mathfrak{Q}_{c,d}^{\varphi,R}\big),\, g\in \dom(T_{\varphi,R}).
\end{equation}
Hence, $\mathfrak{Q}_{c,d}^{\varphi,R}$ is the unique densely defined, closed, symmetric, lower semibounded sesquilinear form uniquely associated to $T_{\varphi,R}$ by the First Representation Theorem.
\end{theorem}
\begin{proof}
The proof of item $(i)$ begins by noting that $\mathfrak{Q}_{c,d}^{\varphi,R}$ is an infinitesimally form bounded perturbation of $\mathfrak{Q}_{c,d}$ by Proposition \ref{p3.3}.  Hence, $\mathfrak{Q}_{c,d}^{\varphi,R}$ is densely defined, closed, and lower semibounded by Proposition \ref{p3.1}.  Furthermore, $\mathfrak{Q}_{c,d}^{\varphi,R}$ is symmetric by inspection.  To prove \eqref{3.26}, let $f\in \dom\big(\mathfrak{Q}_{c,d}^{\varphi,R}\big)=\dom(\mathfrak{Q}_{c,d})$ and $g\in \dom(T_{\varphi,R})$.  Using the boundary conditions for $g$ given in \eqref{23.13.2.27a}, one obtains,
\begin{equation}\lb{3.31cc}
\begin{split}
\wti g^{\,\prime}(a) &= \frac{1}{R_{1,2}}\big[e^{-i\varphi}\wti g(b) - R_{1,1}\wti g(a) \big],\\
\wti g^{\,\prime}(b) &= e^{i\varphi}\big[R_{2,1}\wti g(a) + R_{2,2}\wti g^{\,\prime}(a) \big].
\end{split}
\end{equation}
Therefore, using $\det_{\bbC^2}(R)=1$ and \eqref{3.31cc}, one computes,
\begin{align}
&\ol{\wti f(a)}\wti g^{\,\prime}(a) - \overline{\wti f(b)}\wti g^{\,\prime}(b)    \no\\
&\quad = \frac{\ol{\wti f(a)}}{R_{1,2}}\big\{ e^{-i\varphi}\wti g(b) - R_{1,1} \wti g(a)\big\} - e^{i\varphi}\ol{\wti f(b)}\big\{R_{2,1}\wti g(a) + R_{2,2}\wti g^{\,\prime}(a) \big\}\no\\
&\quad=\frac{\ol{\wti f(a)}}{R_{1,2}}\big\{e^{-i\varphi}\wti g(b) - R_{1,1}\wti g(a) \big\}    \no \\
& \qquad -e^{i\varphi} \ol{\wti f(b)} \bigg\{R_{2,1}\wti g(a) + \frac{R_{2,2}}{R_{1,2}}\big[e^{-i\varphi}\wti g(b) - R_{1,1}\wti g(a) \big] \bigg\}\no\\
&\quad = -\frac{1}{R_{1,2}}\Big\{ R_{1,1}\ol{\wti f(a)}\wti g(a) - e^{-i\varphi}\ol{\wti f(a)}\wti g(b) - e^{i\varphi}\ol{\wti f(b)}\wti g(a) +R_{2,2}\ol{\wti f(b)}\wti g(b)\Big\},    \lb{3.32cc}
\end{align}
after taking a cancellation into account.  The equality in \eqref{3.26} now follows from Lemma \ref{l3.3c} and \eqref{3.32cc}.

To prove item $(ii)$, one notes that $\dom(T_{min})\subseteq \dom\big(\mathfrak{Q}_{c,d}^{\varphi,R}\big)$, so $\mathfrak{Q}_{c,d}^{\varphi,R}$ is densely defined.  Let $\mathfrak{Q}_{c,d}'$ denote the restriction of $\mathfrak{Q}_{c,d}$ to $\dom\big(\mathfrak{Q}_{c,d}^{\varphi,R}\big)$, where the latter domain is defined according to \eqref{3.27}.  Since $\mathfrak{Q}_{c,d}^{\varphi,R}$ is an infinitesimally form bounded perturbation of $\mathfrak{Q}_{c,d}'$ by \eqref{3.7a}, to prove $\mathfrak{Q}_{c,d}^{\varphi,R}$ is closed, it suffices to show that $\mathfrak{Q}_{c,d}'$ is closed.  If $\{f_n\}_{n=1}^{\infty}\subset \dom(\mathfrak{Q}_{c,d}')=\dom\big(\mathfrak{Q}_{c,d}^{\varphi,R}\big)$, $\|f_n-f\|_{\Lr}\to 0$ for some $f\in \Lr$, and $\mathfrak{Q}_{c,d}'(f_n-f_m,f_n-f_m)\to 0$, then the fact that $\mathfrak{Q}_{c,d}$ is closed (cf.~Proposition \ref{p3.1}) implies $f\in \dom(\mathfrak{Q}_{c,d})$ and $\mathfrak{Q}_{c,d}(f_n-f,f_n-f)\to 0$. Using \eqref{3.7a} one obtains 
\begin{align}
\big|\wti f(b)-e^{i\varphi}R_{1,1}\wti f(a)\big|^2 &= \big|\big[\wti f_n(b) - \wti f(b)\big] - e^{i\varphi}R_{1,1}\big[\wti f_n(a) - \wti f(a)\big]\big|^2\lb{3.29}\\
&\leq \mathfrak{Q}_{c,d}(f_n-f,f_n-f) + C_0\|f_n-f\|_{\Lr}^2,\quad n\in \bbN,\no
\end{align}
for some scalar $C_0\in (0,\infty)$ that does not depend on $n\in \bbN$.  Taking $n\to \infty$ throughout \eqref{3.29}, one obtains $\wti f(b) = e^{i\varphi}R_{1,1}\wti f(a)$.  Therefore, $f\in \dom\big(\mathfrak{Q}_{c,d}^{\varphi,R}\big)$, and since $\mathfrak{Q}_{c,d}$ is an extension of $\mathfrak{Q}_{c,d}'$, $\mathfrak{Q}_{c,d}'(f_n-f,f_n-f)\to 0$.  Hence, $\mathfrak{Q}_{c,d}'$ is closed, and it follows that $\mathfrak{Q}_{c,d}^{\varphi,R}$ is closed and lower semibounded.

To verify \eqref{3.28}, let $f\in \dom\big(\mathfrak{Q}_{c,d}^{\varphi,R}\big)$ and $g\in \dom(T_{\varphi,R})$.  Using the relations $\wti f(b) = e^{i\varphi}R_{1,1}\wti f(a)$, $\wti g^{\,\prime}(b) = e^{i\varphi}\big[R_{2,1}\wti g(a) + R_{2,2}\wti g^{\,\prime}(a)\big]$ and $1=\det_{\bbC^2}(R)=R_{1,1}R_{2,2}$, one computes:
\begin{align}
\ol{\wti f(a)}\wti g^{\,\prime}(a) - \ol{\wti f(b)}\wti g^{\,\prime}(b)\lb{3.34cc}&= \ol{\wti f(a)}\wti g^{\,\prime}(a) - e^{-i\varphi}R_{1,1}\ol{\wti f(a)}e^{i\varphi}\big[R_{2,1}\wti g(a) + R_{2,2}\wti g^{\,\prime}(a)\big]\no\\
&\quad = \ol{\wti f(a)}\wti g^{\,\prime}(a) - R_{1,1}R_{2,1}\ol{\wti f(a)}\wti g(a) - R_{1,1}R_{2,2}\ol{\wti f(a)}\wti g^{\,\prime}(a)\no\\
&\quad = - R_{1,1}R_{2,1}\ol{\wti f(a)}\wti g(a).
\end{align}
The equality in \eqref{3.28} now follows from Lemma \ref{l3.3c} and \eqref{3.34cc}.
\end{proof}

\begin{remark}
$(i)$ Since $\mathfrak{Q}_{c,d}$ is independent of the choices of $c\in (a,a_0)$ and $d\in (b_0,b)$ (cf.~Proposition \ref{p3.1}\,$(iii)$), it follows that the sesquilinear forms $\mathfrak{Q}_{c,d}^{\alpha,\beta}$, $\alpha,\beta\in [0,\pi)$, and $\mathfrak{Q}_{c,d}^{\varphi,R}$, $\varphi\in[0,\pi)$, $R\in SL(2,\bbR)$, are also independent of $c$ and $d$.\\[1mm]
$(ii)$ It is clear that the sesquilinear forms for $T_{\alpha,\beta}$ and $T_{\varphi,R}$ in \eqref{3.8}, \eqref{3.10a}, \eqref{3.20aa}, \eqref{3.22aa}, \eqref{3.25}, and \eqref{3.27} depend on the choices of the principal and nonprincipal solutions $u_t(\lambda_0,\dott)$ and $\widehat u_t(\lambda_0,\dott)$, $t\in \{a,b\}$.  However, this is to be expected, as the parametrizations of the self-adjoint extensions of $T_{min}$ given in Proposition \ref{23.t7.t3.12} also depend on the choices of the principal and nonprincipal solutions $u_t(\lambda_0,\dott)$ and $\widehat u_t(\lambda_0,\dott)$, $t\in \{a,b\}$. 
\hfill$\diamond$
\end{remark}

\section{Case Two:  One Limit Circle Endpoint} \lb{s4}

In this section we provide the sesquilinear forms corresponding to the lower semibounded self-adjoint realizations $T_\alpha$ from Proposition \ref{p4.1}. 
We assume, in addition to Hypothesis \ref{h2.1}, that $\tau$ is in the limit circle case at exactly one endpoint of the interval $(a,b)$ and that $T_{min}\geq \lambda_0I_{\Lr}$ for some $\lambda_0\in \bbR$.
For simplicity, we consider the case when $\tau$ is in the limit circle case at $a$ and in the limit point case at $b$.  The situation where $\tau$ is in the limit point case at $a$ and in the limit circle case at $b$ is entirely analogous.  To be precise, we introduce the following hypothesis.

\begin{hypothesis}\lb{h4.1}
In addition to Hypothesis \ref{h2.1}, assume that $\tau$ is in the limit circle case at $a$ and in the limit point case at $b$.  Suppose that $T_{min} \geq \lambda_0 I_{\Lr}$ for some $\lambda_0 \in \bbR$ and that $u_t(\lambda_0,\dott)$ and $\widehat u_t(\lambda_0,\dott)$ are principal and nonprincipal solutions of $\tau u = \lambda_0 u$ on $(a,b)$, respectively, at $t\in \{a,b\}$ that satisfy \eqref{7.3.33AB}.
\end{hypothesis}

Assuming Hypothesis \ref{h4.1}, choose $a_0,b_0\in (a,b)$ such that $a<a_0<b_0<b$ and \eqref{3.1} holds.  Let $c\in (a,a_0)$ and $d\in (b_0,b)$ be fixed.  Next, we formally replace the nonprincipal solution $\widehat{u}_b(\lambda_0,\dott)$ in Section \ref{s3} with the principal solution $u_b(\lambda_0,\dott)$.  More precisely, introducing the differential expressions $N_{\widehat{u}_a(\lambda_0,\dott),c}$ as in \eqref{3.2vv} and $N_{u_b(\lambda_0,\dott),d}$ by
\begin{equation}\lb{4.1vv}
N_{u_b(\lambda_0,\dott),d}g = p^{1/2}u_b(\lambda_0,\dott)\bigg(\frac{g}{u_b(\lambda_0,\dott)}\bigg)',\quad g\in AC_{loc}((d,b)),
\end{equation}
one defines the symmetric sesquilinear form $\mathfrak{Q}_{c,d}$ as follows:
\begin{align}
&\dom(\mathfrak{Q}_{c,d})=\big\{h\in \Lr\,\big|\, h\in AC_{loc}((a,b)),\no\\
&\hspace*{2.3cm} p^{-1/2}h^{[1]}\in L^2((c,d);dx),\, N_{\widehat{u}_a(\lambda_0,\dott),c}h\in L^2((a,c);dx),\lb{4.3}\\
&\hspace*{6.4cm} N_{u_b(\lambda_0,\dott),d}h\in L^2((d,b);dx)\big\},\no
\end{align}
and
\begin{align}
\mathfrak{Q}_{c,d}(f,g) &= \int_a^c dx\, \ol{(N_{\widehat{u}_a(\lambda_0,\dott),c}f)(x)} (N_{\widehat{u}_a(\lambda_0,\dott),c}g)(x) \no\\
& \quad + \int_d^b dx\, \ol{(N_{u_b(\lambda_0,\dott),d}f)(x)} (N_{u_b(\lambda_0,\dott),d}g)(x)     \no\\
&\quad + \lambda_0 \int_a^c r(x) \,dx\, \ol{f(x)}g(x) + \lambda_0\int_d^b r(x) \,dx\, \ol{f(x)}g(x)    \no \\
& \quad + \int_c^d\, dx\, \Big[p(x)^{-1}\ol{f^{[1]}(x)} g^{[1]}(x) + q(x)\ol{f(x)}g(x)\Big]    \no\\
&\quad+ \frac{\widehat{u}_a^{[1]}(\lambda_0,c)}{\widehat{u}_a(\lambda_0,c)}\ol{f(c)}g(c) - \frac{u_b^{[1]}(\lambda_0,d)}{u_b(\lambda_0,d)}\ol{f(d)}g(d),\quad f,g\in \dom(\mathfrak{Q}_{c,d}).\lb{4.4}
\end{align}

Several important properties of the sesquilinear form $\mathfrak{Q}_{c,d}$ are collected in the following result.

\begin{proposition}\lb{p4.2}
Assume Hypothesis \ref{h4.1}.  Let $a<a_0<b_0<b$ with $a_0$ and $b_0$ chosen so that \eqref{3.1} holds and suppose $c\in (a,a_0)$ and $d\in (b_0,b)$.  Then the following statements $(i)$--\,$(vi)$ hold:\\[1mm]
$(i)$ The sesquilinear form $\mathfrak{Q}_{c,d}$ defined by \eqref{4.3} and \eqref{4.4} is densely defined, closed, and lower semibounded in $\Lr$.\\[1mm]
$(ii)$ $\dom(T_{max})\subseteq \dom(\mathfrak{Q}_{c,d})$.\\[1mm]
$(iii)$  If $c'\in (a,a_0)$ and $d'\in (b_0,b)$, then $\mathfrak{Q}_{c,d}=\mathfrak{Q}_{c',d'}$.  That is, the sesquilinear form defined by \eqref{4.3} and \eqref{4.4} is independent of the choices of $c\in (a,a_0)$ and $d\in (b_0,b)$.\\[1mm]
$(iv)$  If $g\in \dom(\mathfrak{Q}_{c,d})$, then the following limit exists:
\begin{equation}
\wti g(a) := \lim_{x \downarrow a} \f{g(x)}{\hatt u_a(\lambda_0,x)}.
\end{equation}
In particular, the generalized boundary value $\wti g(a)$ introduced in \eqref{7.3.34A} for functions in $\dom(T_{max})$ extends to functions in $\dom(\mathfrak{Q}_{c,d})$.\\
$(v)$  If $f\in \dom(\mathfrak{Q}_{c,d})$ and $g\in \dom(T_{max})$, then
\begin{equation}\lb{4.7v}
\lim_{b'\uparrow b} \frac{\ol{f(b')}}{u_b(\lambda_0,b')}W(u_b(\lambda_0,\dott),g)(b')=0.
\end{equation}
$(vi)$  For every $\varepsilon>0$ there exists $C(\varepsilon)>0$ such that
\begin{equation}\lb{4.8vvv}
\big|\wti f(a)\big|^2 \leq \varepsilon\mathfrak{Q}_{c,d}(f,f) + C(\varepsilon)\|f\|_{\Lr}^2,\quad f\in \dom(\mathfrak{Q}_{c,d}).
\end{equation}
\end{proposition}

\begin{remark}
The properties of $\mathfrak{Q}_{c,d}$ summarized in Proposition \ref{p4.2} are discussed in detail in \cite{{BHS20}} (see \cite[Lemma 6.9.4, Corollary 6.12.2, Lemma 6.12.3, Proof of Lemma 6.12.5]{BHS20}) and \eqref{4.8vvv} is entirely analogous to Proposition \ref{p3.3}. For the connection with \cite[Sect.~ 6.12]{BHS20}, see Appendix \ref{{sA}}.    \hfill$\diamond$
\end{remark}

\begin{lemma}\lb{l4.4}
Assume Hypothesis \ref{h4.1}.  Let $a<a_0<b_0<b$ with $a_0$ and $b_0$ chosen so that \eqref{3.1} holds and suppose $c\in (a,a_0)$ and $d\in (b_0,b)$.  If $f\in \dom(\mathfrak{Q}_{c,d})$ and $g\in \dom(T_{max})$, then
\begin{align}
(f,T_{max}g)_{\Lr} = \mathfrak{Q}_{c,d}(f,g) + \overline{\wti f(a)}\wti g^{\,\prime}(a).\lb{4.8v}
\end{align}
\end{lemma}
\begin{proof}
Repeating the calculations in \eqref{3.9m}--\eqref{3.13c} with $u_b(\lambda_0,\dott)$ in place of $\widehat{u}_b(\lambda_0,\dott)$, one obtains for $f\in \dom(\mathfrak{Q}_{c,d})$ and $g\in \dom(T_{max})$,
\begin{align}
&(f,T_{max}g)_{\Lr}\no\\
&\quad = \lim_{a'\downarrow a}\Bigg\{-\frac{\ol{f}}{\widehat{u}_a(\lambda_0,\dott)}\Bigg[p\widehat{u}_a(\lambda_0,\dott)^2\bigg(\frac{g}{\widehat{u}_a(\lambda_0,\dott)}\bigg)'\Bigg]\Bigg|_{a'}^c\no\\
&\hspace*{1.8cm} + \int_{a'}^cdx\, \bigg(\frac{\overline{f}}{\widehat{u}_a(\lambda_0,\dott)}\bigg)' p \widehat{u}_a(\lambda_0,\dott)^2\bigg(\frac{g}{\widehat{u}_a(\lambda_0,\dott)}\bigg)' \Bigg\}\no\\
&\qquad + \lambda_0 \int_a^c r\,dx\, \overline{f}g - \overline{f}g^{[1]}\big|_c^d + \int_c^d dx\, \big(p^{-1}\overline{f^{[1]}}g^{[1]} + q\overline{f}g\big) + \lambda_0 \int_d^b r\,dx\, \overline{f}g\no\\
&\qquad + \lim_{b'\uparrow b}\Bigg\{-\frac{\ol{f}}{u_b(\lambda_0,\dott)}\Bigg[pu_b(\lambda_0,\dott)^2\bigg(\frac{g}{u_b(\lambda_0,\dott)}\bigg)'\Bigg]\Bigg|_d^{b'}\no\\
&\hspace{2cm} + \int_d^{b'}dx\, \bigg(\frac{\overline{f}}{u_b(\lambda_0,\dott)}\bigg)' p u_b(\lambda_0,\dott)^2\bigg(\frac{g}{u_b(\lambda_0,\dott)}\bigg)' \Bigg\}.\lb{4.7}
\end{align}
In analogy with \eqref{3.12c}, the evaluation terms at $c$ and $d$ in \eqref{4.7} are
\begin{equation}\lb{4.10v}
\frac{\widehat{u}_a^{[1]}(\lambda_0,c)}{\widehat{u}_a(\lambda_0,c)}\ol{f(c)}g(c) - \frac{u_b^{[1]}(\lambda_0,d)}{u_b(\lambda_0,d)}\ol{f(d)}g(d).
\end{equation}
Moreover, \eqref{3.14c} remains valid.  However, in lieu of \eqref{3.15c}, one now obtains, as a consequence of \eqref{4.7v},
\begin{equation}\lb{4.11v}
\begin{split}
&\lim_{b'\uparrow b} \Bigg[\frac{\overline{f}}{u_b(\lambda_0,\dott)}pu_b(\lambda_0,\dott)^2\bigg(\frac{g}{u_b(\lambda_0,\dott)}\bigg)'\Bigg](b')\\
&\quad = \lim_{b'\uparrow b}\frac{\ol{f(b')}}{u_b(\lambda_0,b')}W(u_b(\lambda_0,\dott),g)(b') = 0.
\end{split}
\end{equation}
Hence, \eqref{4.8v} follows by combining \eqref{4.7}, \eqref{4.10v}, and \eqref{4.11v}.
\end{proof}

In the next theorem we provide the sesquilinear form corresponding to the self-adjoint extensions $T_{\alpha}$, $\alpha\in[0,\pi)$, of $T_{min}$ with a separated boundary condition from Proposition \ref{p4.1}\,$(ii)$.

\begin{theorem}\label{caselp}
Assume Hypothesis \ref{h4.1}.  Let $a<a_0<b_0<b$ with $a_0$ and $b_0$ chosen so that \eqref{3.1} holds and suppose $c\in (a,a_0)$ and $d\in (b_0,b)$.  Then the following statements $(i)$ and $(ii)$ hold:\\[1mm]
$(i)$  If $\alpha\in (0,\pi)$, then the sesquilinear form $\mathfrak{Q}_{c,d}^{\alpha}$ defined by
\begin{align}
\begin{split} 
\mathfrak{Q}_{c,d}^{\alpha}(f,g) = \mathfrak{Q}_{c,d}(f,g)-\cot(\alpha)\overline{\widetilde{f}(a)}\widetilde{g}(a),&   \\ 
f,g\in \dom\big(\mathfrak{Q}_{c,d}^{\alpha}\big)=\dom(\mathfrak{Q}_{c,d}),&    \lb{4.8}
\end{split} 
\end{align}
is densely defined, closed, symmetric, and lower semibounded.  In addition,
\begin{equation}\lb{4.9}
(f,T_{\alpha}g)_{\Lr} = \mathfrak{Q}_{c,d}^{\alpha}(f,g),\quad f\in \dom\big(\mathfrak{Q}_{c,d}^{\alpha}\big),\, g\in \dom(T_{\alpha}).
\end{equation}
Hence, $\mathfrak{Q}_{c,d}^{\alpha}$ is the unique densely defined, closed, symmetric, lower semibounded sesquilinear form uniquely associated to $T_{\alpha}$ by the First Representation Theorem.\\[1mm]
$(ii)$  If $\alpha=0$, then the sesquilinear form $\mathfrak{Q}_{c,d}^{0}$ defined by
\begin{align}
\mathfrak{Q}_{c,d}^{0}(f,g) = \mathfrak{Q}_{c,d}(f,g), \quad f,g\in \dom(\mathfrak{Q}_{c,d}^{0})= \big\{h\in \dom(\mathfrak{Q}_{c,d})\,\big|\, \wti h(a)=0\big\},\lb{4.10}
\end{align}
is densely defined, closed, symmetric, and lower semibounded.  In addition,
\begin{equation}\lb{4.11}
(f,T_0g)_{\Lr} = \mathfrak{Q}_{c,d}^{0}(f,g),\quad f\in \dom\big(\mathfrak{Q}_{c,d}^0\big),\, g\in \dom(T_0).
\end{equation}
Hence, $\mathfrak{Q}_{c,d}^0$ is the unique densely defined, closed, symmetric, lower semibounded sesquilinear form uniquely associated to $T_0$ by the First Representation Theorem.
\end{theorem}
\begin{proof}
$(i)$  That $\mathfrak{Q}_{c,d}^{\alpha}$ is closed and lower semibounded follows from the infinitesimal boundedness property summarized in \eqref{4.8vvv}.  It is clear by inspection that $\mathfrak{Q}_{c,d}^{\alpha}$ is symmetric, and $\dom(T_{min})\subseteq \dom\big(\mathfrak{Q}_{c,d}^{\alpha}\big)$ shows that $\mathfrak{Q}_{c,d}^{\alpha}$ is densely defined in $\Lr$.  If $f\in \dom\big(\mathfrak{Q}_{c,d}^{\alpha}\big)$ and $g\in \dom(T_{\alpha})$, then \eqref{4.8v} and the boundary condition $\wti g^{\,\prime}(a) = -\cot(\alpha)\wti g(a)$ yield:
\begin{equation}
\begin{split}
(f,T_{\alpha}g)_{\Lr} &= (f,T_{max}g)_{\Lr}\\
&= \mathfrak{Q}_{c,d}(f,g) - \cot(\alpha)\ol{\wti f(a)}\wti g(a) = \mathfrak{Q}_{c,d}^{\alpha}(f,g).
\end{split}
\end{equation}

\noindent
$(ii)$  One notes that $\mathfrak{Q}_{c,d}^0$ is densely defined since $\dom(T_{min})\subseteq \dom\big(\mathfrak{Q}_{c,d}^0\big)$, and $\mathfrak{Q}_{c,d}^0$ is lower semibounded since it is a restriction of $\mathfrak{Q}_{c,d}^{\pi/2}$, and the latter is lower semibounded by part $(i)$. To prove that $\mathfrak{Q}_{c,d}^0$ is closed, let $\{f_n\}_{n=1}^{\infty}\subset \dom\big(\mathfrak{Q}_{c,d}^0\big)$ be a sequence such that $\|f_n-f\|_{\Lr}\to 0$ for some $f\in \Lr$ and $\mathfrak{Q}_{c,d}^0(f_n-f_m,f_n-f_m)\to 0$.  Since $\mathfrak{Q}_{c,d}^0$ is a restriction of $\mathfrak{Q}_{c,d}^{\pi/2}$, and the latter is closed, it follows that $f\in \dom\big(\mathfrak{Q}_{c,d}^{\pi/2}\big)$ and $\mathfrak{Q}_{c,d}^{\pi/2}(f_n-f,f_n-f)\to 0$.  By \eqref{4.8vvv}, one obtains: For every $\varepsilon>0$, there exists $\widehat{C}(\varepsilon)>0$ such that
\begin{equation}\lb{3.23u}
\big|\wti g(a)\big|^2 \leq \varepsilon \mathfrak{Q}_{c,d}^{\pi/2}(g,g) + \widehat{C}(\varepsilon)\|g\|_{\Lr}^2,\quad g\in \dom\big(\mathfrak{Q}_{c,d}^{\pi/2}\big).
\end{equation}
In turn, \eqref{3.23u} with $\varepsilon = 1$ yields:
\begin{equation}\lb{3.24u}
\begin{split}
\big|\wti f(a)\big|^2 &= \big|\wti f_n(a)-\wti f(a)\big|^2\\
&\leq \mathfrak{Q}_{c,d}^{\pi/2}(f_n-f,f_n-f) + \widehat{C}(1)\|f_n-f\|_{\Lr}^2,\quad n\in \bbN.
\end{split}
\end{equation}
Taking $n\to \infty$ throughout \eqref{3.24u} yields $\wti f(a)=0$, thereby implying $f\in \dom\big(\mathfrak{Q}_{c,d}^0\big)$.  Using once more that $\mathfrak{Q}_{c,d}^0$ is a restriction of $\mathfrak{Q}_{c,d}^{\pi/2}$, it follows that $\mathfrak{Q}_{c,d}^0(f_n-f,f_n-f)\to 0$.  Hence, one concludes that $\mathfrak{Q}_{c,d}^0$ is closed.  Finally, \eqref{4.11} follows from \eqref{4.8v} and the boundary condition $\wti f(a)=0$.
\end{proof}

\section{Case Three: Two Limit Point Endpoints} \lb{s5}

In this final section we 
provide the sesquilinear form corresponding to the unique, lower semibounded, self-adjoint realization 
from Proposition \ref{p2.14}. We assume, in addition to Hypothesis \ref{h2.1}, that $\tau$ is in the limit point case at both endpoints of the interval $(a,b)$ and that $T_{min}\geq \lambda_0I_{\Lr}$ for some $\lambda_0\in \bbR$.  To be precise, we introduce the following hypothesis.

\begin{hypothesis}\lb{h5.1}
In addition to Hypothesis \ref{h2.1}, assume that $\tau$ is in the limit point case at both $a$ and $b$.  Suppose that $T_{min} \geq \lambda_0 I_{\Lr}$ for some $\lambda_0 \in \bbR$ and that $u_t(\lambda_0,\dott)$ and $\widehat u_t(\lambda_0,\dott)$ are principal and nonprincipal solutions of $\tau u = \lambda_0 u$ on $(a,b)$, respectively, at $t\in \{a,b\}$ that satisfy \eqref{7.3.33AB}.
\end{hypothesis}

Under Hypothesis \ref{h5.1}, the operator $T:=T_{min}=T_{max}$ is self-adjoint (equivalently, $\dot T$ is essentially self-adjoint) by Proposition \ref{p2.14}.  In particular, $T_{min}$ is self-adjoint and possesses no nontrivial self-adjoint extension.

Assuming Hypothesis \ref{h5.1}, choose $a_0,b_0\in (a,b)$ such that $a<a_0<b_0<b$ and \eqref{3.1} holds.  Let $c\in (a,a_0)$ and $d\in (b_0,b)$ be fixed.  Next, we formally replace the nonprincipal solutions $\widehat{u}_t(\lambda_0,\dott)$, $t\in\{a,b\}$, in Section \ref{s3} with the principal solutions $u_t(\lambda_0,\dott)$, $t\in\{a,b\}$.  More precisely, introducing the differential expressions $N_{u_b(\lambda_0,\dott),d}$ as in \eqref{4.1vv} and $N_{u_a(\lambda_0,\dott),c}$ by
\begin{equation}
N_{u_a(\lambda_0,\dott),c}g = p^{1/2}u_a(\lambda_0,\dott)\bigg(\frac{g}{u_a(\lambda_0,\dott)}\bigg)',\quad g\in AC_{loc}((a,c)),
\end{equation}
one defines the symmetric sesquilinear form $\mathfrak{Q}_{c,d}$ as follows:
\begin{align}
&\dom(\mathfrak{Q}_{c,d})=\big\{h\in \Lr\,\big|\, h\in AC_{loc}((a,b)),\no\\
&\hspace*{2.3cm} p^{-1/2}h^{[1]}\in L^2((c,d);dx),\, N_{u_a(\lambda_0,\dott),c}h\in L^2((a,c);dx),\lb{5.2}\\
&\hspace*{6.4cm} N_{u_b(\lambda_0,\dott),d}h\in L^2((d,b);dx)\big\},\no
\end{align}
and
\begin{align}
\mathfrak{Q}_{c,d}(f,g) &= \int_a^c dx\, \ol{(N_{u_a(\lambda_0,\dott),c}f)(x)} (N_{u_a(\lambda_0,\dott),c}g)(x)    \no \\
& \quad + \int_d^b dx\, \ol{(N_{u_b(\lambda_0,\dott),d}f)(x)} (N_{u_b(\lambda_0,\dott),d}g)(x)      \no\\
&\quad + \lambda_0 \int_a^c r(x) \,dx\, \ol{f(x)}g(x) + \lambda_0\int_d^b r(x) \,dx\, \ol{f(x)}g(x)    \no \\
& \quad + \int_c^d\, dx\, \Big[p(x)^{-1}\ol{f^{[1]}(x)} g^{[1]}(x) + q(x) \ol{f(x)}g(x)\Big]     \no\\
&\quad+ \frac{u_a^{[1]}(\lambda_0,c)}{u_a(\lambda_0,c)}\ol{f(c)}g(c) - \frac{u_b^{[1]}(\lambda_0,d)}{u_b(\lambda_0,d)}\ol{f(d)}g(d),\quad f,g\in \dom(\mathfrak{Q}_{c,d}).\lb{5.3}
\end{align}

Several important properties of the sesquilinear form $\mathfrak{Q}_{c,d}$ are collected in the following result.

\begin{proposition}\lb{p5.2}
Assume Hypothesis \ref{h5.1}.  Let $a<a_0<b_0<b$ with $a_0$ and $b_0$ chosen so that \eqref{3.1} holds and suppose $c\in (a,a_0)$ and $d\in (b_0,b)$.  Then the following statements $(i)$--\,$(iv)$ hold:\\[1mm]
$(i)$ The sesquilinear form $\mathfrak{Q}_{c,d}$ defined by \eqref{5.2} and \eqref{5.3} is densely defined, closed, and lower semibounded in $\Lr$.\\[1mm]
$(ii)$ $\dom(T_{max})\subseteq \dom(\mathfrak{Q}_{c,d})$.\\[1mm]
$(iii)$  If $c'\in (a,a_0)$ and $d'\in (b_0,b)$, then $\mathfrak{Q}_{c,d}=\mathfrak{Q}_{c',d'}$.  That is, the sesquilinear form defined by \eqref{5.2} and \eqref{5.3} is independent of the choices of $c\in (a,a_0)$ and $d\in (b_0,b)$.\\[1mm]
$(iv)$  If $f\in \dom(\mathfrak{Q}_{c,d})$ and $g\in \dom(T_{max})$, then
\begin{equation}\lb{5.4}
\lim_{a'\downarrow a} \frac{\ol{f(a')}}{u_a(\lambda_0,a')}W(u_a(\lambda_0,\dott),g)(a')=\lim_{b'\uparrow b} \frac{\ol{f(b')}}{u_b(\lambda_0,b')}W(u_b(\lambda_0,\dott),g)(b')=0.
\end{equation}
\end{proposition}

\begin{remark}
The proofs of items $(i)$--\,$(iv)$ in Proposition \ref{p5.2} are entirely analogous to those of the corresponding facts in Proposition \ref{p4.2}. \hfill$\diamond$
\end{remark}

\begin{theorem}\label{caselplp}
Assume Hypothesis \ref{h5.1}.  Let $a<a_0<b_0<b$ with $a_0$ and $b_0$ chosen so that \eqref{3.1} holds and suppose $c\in (a,a_0)$ and $d\in (b_0,b)$.  If $T:=T_{min}=T_{max}$, then
\begin{align}\lb{5.5}
(f,Tg)_{\Lr} = \mathfrak{Q}_{c,d}(f,g),\quad f\in \dom(\mathfrak{Q}_{c,d}),\, g\in \dom(T).
\end{align}
Hence, $\mathfrak{Q}_{c,d}$ is the unique densely defined, closed, symmetric, lower semibounded sesquilinear form uniquely associated to $T$ by the First Representation Theorem.
\end{theorem}
\begin{proof}
Repeating the calculations in \eqref{4.8} with $u_a(\lambda_0,\dott)$ in place of $\widehat{u}_a(\lambda_0,\dott)$, one obtains for $f\in \dom(\mathfrak{Q}_{c,d})$ and $g\in \dom(T)$,
\begin{align}
&(f,T_{max}g)_{\Lr}\no\\
&\quad = \lim_{a'\downarrow a}\Bigg\{-\frac{\ol{f}}{u_a(\lambda_0,\dott)}\Bigg[pu_a(\lambda_0,\dott)^2\bigg(\frac{g}{u_a(\lambda_0,\dott)}\bigg)'\Bigg]\Bigg|_{a'}^c\no\\
&\hspace*{1.8cm} + \int_{a'}^cdx\, \bigg(\frac{\overline{f}}{u_a(\lambda_0,\dott)}\bigg)' p u_a(\lambda_0,\dott)^2\bigg(\frac{g}{u_a(\lambda_0,\dott)}\bigg)' \Bigg\}\no\\
&\qquad + \lambda_0 \int_a^c r\,dx\, \overline{f}g - \overline{f}g^{[1]}\big|_c^d + \int_c^d dx\, \big(p^{-1}\overline{f^{[1]}}g^{[1]} + q\overline{f}g\big) + \lambda_0 \int_d^b r\,dx\, \overline{f}g\no\\
&\qquad + \lim_{b'\uparrow b}\Bigg\{-\frac{\ol{f}}{u_b(\lambda_0,\dott)}\Bigg[pu_b(\lambda_0,\dott)^2\bigg(\frac{g}{u_b(\lambda_0,\dott)}\bigg)'\Bigg]\Bigg|_d^{b'}\no\\
&\hspace{2cm} + \int_d^{b'}dx\, \bigg(\frac{\overline{f}}{u_b(\lambda_0,\dott)}\bigg)' p u_b(\lambda_0,\dott)^2\bigg(\frac{g}{u_b(\lambda_0,\dott)}\bigg)' \Bigg\}.\lb{5.6}
\end{align}
In analogy with \eqref{4.10v}, the evaluation terms at $c$ and $d$ in \eqref{5.6} are
\begin{equation}\lb{5.7}
\frac{u_a^{[1]}(\lambda_0,c)}{u_a(\lambda_0,c)}\ol{f(c)}g(c) - \frac{u_b^{[1]}(\lambda_0,d)}{u_b(\lambda_0,d)}\ol{f(d)}g(d).
\end{equation}
Moreover, \eqref{4.11v} remains valid.  In addition, as a consequence of \eqref{5.4},
\begin{equation}\lb{5.8}
\begin{split}
&\lim_{a'\downarrow a} \Bigg[\frac{\overline{f}}{u_a(\lambda_0,\dott)}pu_a(\lambda_0,\dott)^2\bigg(\frac{g}{u_a(\lambda_0,\dott)}\bigg)'\Bigg](a')\\
&\quad = \lim_{a'\downarrow a}\frac{\ol{f(a')}}{u_a(\lambda_0,a')}W(u_a(\lambda_0,\dott),g)(a') = 0.
\end{split}
\end{equation}
Hence, \eqref{5.5} follows by combining \eqref{5.6}, \eqref{5.7}, and \eqref{5.8}.
\end{proof}

\appendix
\section{Approach via Boundary Triplets and Boundary Pairs} \label{{sA}}

In this appendix we briefly provide the background of the results in Section \ref{s3} and Section \ref{s4}
of this paper in terms of the boundary triplets and boundary pairs 
following the extensive 
treatment in \cite[Chs.~2, 5, 6]{BHS20}. 
By means of boundary pairs one can systematically treat the semibounded forms
that are associated with the lower semibounded self-adjoint  extensions of lower semibounded
symmetric operators. 
In this paper inner products and sesquilinear forms are linear 
in the second entry and anti-linear in the first entry; 
in the references to \cite{BHS20} one should be aware of the present convention.
Thus, when a sesquilinear form $\st$ in a Hilbert space $\sH$ 
is densely defined, closed, and lower semibounded, then there exists a unique 
self-adjoint operator $H$ in $\sH$, such that
\[
 \st[f,g]=(f, g)_\sH=(f, Hg)_\sH, \quad f \in \dom (\st), \,\, g \in \dom (H) \subset \dom (\st),
\]
by the First Representation Theorem. The notation $\st=\st_H$ is used to indicate the connection with $H$.

\medskip
\noindent
\textbf{Boundary Triplets.}
Let $S$ be a closed densely defined symmetric operator in a Hilbert space $\sH$
and assume that the defect numbers of $S$ are equal to $(n,n)$, $n\in\N$.   
A triplet $\{\bbC^n,\Gamma_0,\Gamma_1\}$ is called
a {\it boundary triplet} for $S^*$ if
the linear mappings $\Gamma_0,\Gamma_1:\dom (S^*)\rightarrow\bbC^n$
satisfy the abstract Green identity,
\begin{equation*}
 (f, S^*g)_\sH-(S^*f,g)_\sH= (\Gamma_0 f, \Gamma_1 g)_{\bbC^n}-
  (\Gamma_1 f , \Gamma_0 g)_{\bbC^n},
   \quad f,g\in\dom (S^*),
\end{equation*}
and $(\Gamma_0,\Gamma_1)^\top:\dom (S^*)\rightarrow\bbC^{2n}$ is onto,
see \cite[Definition 2.1.1]{BHS20}.
If $\{\bbC^n,\Gamma_0,\Gamma_1\}$ is a boundary triplet for $S^*$, then
one has
\begin{equation*} 
 \dom (S)=\{ g \in \dom (S^*) \,|\, \Gamma_0 g=\Gamma_1 g=0 \}
 \end{equation*}
and the mapping $(\Gamma_0,\Gamma_1)^\top:\dom (S^*)\rightarrow\bbC^{2n}$
is continuous if $\dom (S^*)$ is equipped with the graph norm.
The self-adjoint  extensions $A_\Theta$ of $S$ are parametrized
over the self-adjoint relations (multi-valued operators) $\Theta$ in $\bbC^n$ via  
\begin{equation}\label{bt000}
A_\Theta g= S^* g, \quad
g\in \dom  (A_\Theta)
=\big\{ h \in \dom (S^*) \,|\,  \{\Gamma_0 h, \Gamma_1 h \} \in \Theta \big\},
\end{equation}
see \cite[Theorem 2.1.3]{BHS20}.  
We note that if $\Theta$ is a self-adjoint relation in $\bbC^n$, then
$\dom (\Theta)=(\mul (\Theta))^\perp$ and one has the decomposition
$\bbC^n=\dom (\Theta) \oplus \mul (\Theta)$. In this context we recall that the multi-valued part $\mul (\Theta)$ is given by $\{ h \in \bbC^n\,|\, \{0,h\} \in \Theta \}$. Let $P$ be the orthogonal projection
onto $\dom (\Theta)$ and define the orthogonal operator part $\Theta_{\rm op}=P \Theta$.
Then there is the componentwise orthogonal decomposition
\begin{equation}\label{thetaop}
 \Theta =\Theta_{\rm op} \, \widehat \oplus \, (\{0\} \times \mul (\Theta)),
\end{equation}
where $\Theta_{\rm op}$ is a self-adjoint operator in $\dom (\Theta)$
and the second summand in the  right-hand side is a purely multi-valued self-adjoint relation in $\mul (\Theta)$.

\medskip
\noindent
\textbf{Boundary Pairs.} 
Assume in addition that the closed densely defined symmetric operator $S$ with defect numbers $(n,n)$
 is lower semibounded.  
In this case all self-adjoint extensions of $S$ are lower semibounded.
Recall that the form $\sss[f,g]=(f,Sg)$, $f,g \in \dom (S)$, is closable and that the Friedrichs extension $S_{\rm F}$ of $S$
is the unique self-adjoint operator that is associated with the closure $\overline{\sss}\,(=\st_{S_{\rm F}})$
 via the First Representation Theorem.
Moreover, let $S_1$ be a self-adjoint extension of $S$ which satisfies
\begin{equation}\label{mark}
\dom (S^*) \subseteq \dom (\st_{S_1}), 
\end{equation}
where $\st_{S_1}$ is the closed semibounded form associated with $S_1$ via the First Representation Theorem.
The condition \eqref{mark} is equivalent to
\begin{equation}\label{formdomS1}
 \dom (\st_{S_1})= \ker (S^*-c I_{\sH})  \dot{+} \dom (\st_{S_{\rm F}}), \quad \textrm{a direct sum},
\end{equation}
where $c$ is below the lower bound of $S_1$, and due to finite defect,
\eqref{mark} is also equivalent to the simple condition
\begin{equation*}
 \dom (S)=\dom (S_{\rm F}) \cap \dom (S_1),
\end{equation*}
see \cite[Theorem 5.3.8]{BHS20}.
The next lemma involves the notion of a boundary pair for $S$ with finite defect numbers;
see \cite[Lemma~5.6.5]{BHS20} for the general case.

\begin{lemma}\label{bdpairlemma}
Let $\{\bbC^n, \Gamma_0, \Gamma_1\}$ be an arbitrary
boundary triplet for $S^*$ and let $S_1$ be a self-adjoint extension of $S$
which satisfies \eqref{mark}. Let $\Lambda : \dom (\st_{S_1}) \to \bbC^n$
be a linear mapping which is bounded when
$\dom (\st_{S_1})$ is provided with the inner product associated with $\st_{S_1}-c$,
where $c$ is below the lower bound of $S_1$.
If $\Lambda$ extends $\Gamma_0$, then the self-adjoint extension $S_0$, $\dom (S_0)=\ker (\Gamma_0)$,
coincides with the Friedrichs extension $S_{\rm F}$ and the following equalities hold:
\begin{equation*}
  \ker (\Lambda)=\dom (\st_{S_{\rm F}}) \, \text{ and } \, \ran (\Lambda) =\bbC^n.
\end{equation*}
\end{lemma}

\begin{proof}
Since $\Lambda$ extends $\Gamma_0$, one concludes that 
$\ran (\Lambda)=\ran (\Gamma_0)=\bbC^n$
and also $\dom (S_0)=\ker (\Gamma_0) \subseteq \ker (\Lambda)$.
In particular, $\dom (S) \subseteq \ker (\Lambda)$ and hence by continuity of $\Lambda$ and
the definition of the Friedrichs extension $S_{\rm F}$ 
one concludes that $\dom (\st_{S_{\rm F}}) \subseteq \ker (\Lambda)$.
On the other hand, since the sum in \eqref{formdomS1} is direct
and $\dim (\ker (S^*-c I_{\sH})) = n < \infty$ it follows that $\ker (\Lambda)=\dom (\st_{S_{\rm F}})$
and that $\Lambda$ maps $\ker (S^*-c I_{\sH})$ bijectively onto $\bbC^n$.
Combining this with the stated inclusion 
$\dom (S_0) \subseteq \ker (\Lambda)$ gives $\dom (S_0) \subseteq \dom (\st_{S_{\rm F}})$.
This implies that $S_0=S_{\rm F}$ by \cite[Theorem~5.3.3]{BHS20}.
\end{proof}

\medskip

The pair $\{\bbC^n, \Lambda\}$, where $\Lambda : \dom (\st_{S_1}) \to \bbC^n$ is bounded in the form topology
on $\st_{S_1}$ is called a \textit{boundary pair} for $S$ if $\ker (\Lambda) = \dom (\st_{S_{\rm F}})$,
see \cite[Definition 5.6.1]{BHS20}.
If, in addition, $\dom (S_1) = \ker (\Gamma_1)$,
then $\{\bbC^n, \Gamma_0, \Gamma_1\}$ and $\{\bbC^n, \Lambda\}$
are \textit{compatible} corresponding to $S_1$,
see \cite[Definition 5.6.4]{BHS20}, and the identity
\begin{equation*}
  (f,S^* g)_{\sH} = \st_{S_1}[f,g]+(\Lambda f, \Gamma_1 g)_{\bbC^n},
\quad f \in \dom (\st_{S_1}), \,\, g \in \dom (S^*), 
\end{equation*}
holds, see \cite[Corollary 5.6.7]{BHS20}.
Hence, Lemma~\ref{bdpairlemma} offers general sufficient conditions needed to construct
a compatible boundary pair $\{\bbC^n, \Lambda\}$ for $S$ corresponding to $S_1$.
Boundary pairs offer a general tool to describe forms generated by semibounded self-adjoint extensions
of lower semibounded symmetric operators via boundary conditions.

\medskip

Now let $\{\bbC^n, \Lambda\}$ be a compatible boundary pair corresponding to $S_1$.
Then the closed semibounded form $\st_\Theta$ associated with the self-adjoint extension $A_\Theta$
can be expressed in terms of the form $\st_{S_1}$ and the boundary pair $\{\bbC^n, \Lambda\}$ as follows
\begin{align}\label{ttt1}
\begin{split}
&  \st_\Theta[f,g]= \st_{S_1} [f,g]+(\Lambda f, \Theta_{\rm op}  \Lambda g)_{\bbC^n},  \\[1mm]
& f,g \in  \dom (\st_\Theta)=\{ h \in \dom (\st_{S_1}) \,|\, \Lambda h \in \dom (\Theta_{\rm op}) \},
\end{split}
\end{align}
see \cite[Corollary 5.6.14]{BHS20}. Hence, if $\Theta$ is a matrix, then \eqref{ttt1} reads
\begin{equation}\label{ttt1a}
\st_\Theta[f,g]= \st_{S_1} [f,g]+(\Lambda f, \Theta \Lambda g)_{\bbC^n}, 
 \quad   f,g \in  \dom (\st_\Theta)=\dom (\st_{S_1}).
\end{equation}
Moreover, if $\mul (\Theta)=\bbC^n$, then
\begin{equation*}
\st_\Theta \subseteq \st_{S_1}, \quad
\dom (\st_\Theta)=\{ h \in \dom (\st_{S_1}) \,|\, \Lambda h =0 \},
\end{equation*}
which corresponds to the Friedrichs extension.  In particular, for the case
$n=1$ one has $\Theta \in \bbR \cup \{\infty\}$.
One notes that for $\Theta \in \bbR$ the decomposition reads   
 \begin{equation*} 
 \st_\Theta[f,g]= \st_{S_1}[f,g] + (\Lambda f, \Theta \Lambda g)_{\bbC}, 
\quad f,g \in \dom (\st_\Theta)=\dom (\st_{S_1}),
\end{equation*}
while for $\Theta=\infty$ one has
\begin{equation*}
\st_\Theta \subseteq \st_{S_1}, \quad
\dom (\st_\Theta)=\{ h \in \dom (\st_{S_1}) \,|\, \Lambda h=0\}.
\end{equation*}

\medskip
\noindent
\textbf{Self-adjoint Linear Relations in $\bbC^n$.}
The structure of the self-adjoint extensions in \eqref{bt000} is clarified next.
It follows from \cite[Theorem 1.10.5, Corollary 1.10.8, Proposition 1.10.3]{BHS20} 
that any self-adjoint relation $\Theta$ in $\bbC^n$ can be expressed as
\begin{equation}\label{tes2}
\Theta=\left\{ \{{\bf u}, {\bf v}\} \in \bbC^n \times \bbC^n\,|\,  \cB \bf u = \cA \bf v \right\},
\end{equation}
where the $n \times n$ matrices $\cA$ and $\cB$ satisfy
\begin{equation}\label{tes1}
  \cA\cB^*=\cB \cA^*, \quad \rank (\cB \;\; \cA)=n,
\end{equation}
and $(\cB \;\; \cA)$ stands for the $n \times 2n$ matrix of the columns of $\cB$ and $\cA$.
The multi-valued part of $\Theta$ is given by
\begin{equation*}
 \mul (\Theta)= \left\{ {\bf v}  \in \bbC^n \,| \, \cA \bf v =0 \right\}= \ker (\cA),
\end{equation*}
so that it follows from \eqref{thetaop} and \eqref{tes2} that
\begin{equation}\label{oper}
 \cB {\bf u}= \cA \Theta_{\rm op} {\bf u}, \quad {\bf u} \in \dom (\Theta) =(\mul (\Theta))^\perp=(\ker (\cA))^\perp=\ran (\cA^*).
\end{equation}
Therefore, $\Theta_{\rm op}$ can be expressed as
\begin{equation}\label{oper1}
 \Theta_{\rm op}=\cA^{[-1]}\cB \upharpoonright \ran (\cA^*),
\end{equation}
where $\cA^{[-1]}$ stands for the Moore--Penrose inverse of $\cA$.
Hence, if $\ker (\cA)=\{0\}$, then  $\dom (\Theta) =\bbC^n$ and $\Theta=\cA^{-1}\cB$ 
is an $n \times n$ self-adjoint matrix.
Moreover, if $\ker (\cA)=\bbC^n$, then $\dom (\Theta)=\{0\}$ and $\Theta$
is a purely multi-valued self-adjoint relation in $\bbC^n$ given by  $\Theta=\{0\} \times \bbC^n$.
 
In the case $n=2$ and $\dim (\ker (\cA))=1$ the selfadjoint operator $\Theta_{\rm op}$,
acting in the invariant one-dimensional subspace $\dom (\Theta)$, 
is multiplication by the unique real number $c_\Theta$ given by
\begin{equation}\label{tes4}
  \cB {\bf u} =c_\Theta \cA {\bf u}, \quad {\bf u} \in \dom (\Theta)=(\ker (\cA))^\perp, \quad {\bf u} \neq 0.
\end{equation}
In the case $n=1$  the self-adjoint relation $\Theta$ can be expressed as
\begin{equation}\label{tes00}
 \Theta= \left\{ \{\bf u,  \bf v \}  \in \bbC \times \bbC  \,|\,
 \cos (\gamma)  \bf u + \sin (\gamma) \bf v =0 \right\},
\end{equation}
with $\gamma \in [0, \pi)$.
If $\gamma=0$, then $\mul (\Theta)=\bbC$ and $\Theta=\{0\} \times \bbC$,
whereas if $\gamma \neq 0$, then $\mul (\Theta)=\{0\}$ and $\Theta=\Theta_{\rm op}$
is multiplication by $-\cot (\gamma)$.
 
Summarizing, for a pair of $n \times n$ matrices $\cA$ and $\cB$ satisfying \eqref{tes1}  
and $\Theta$ given by \eqref{tes2}, the self-adjoint extension $A_\Theta$ of $S$ in \eqref{bt000} 
is given by
\begin{equation}\label{bt000+}
A_{\Theta} g= S^* g, \quad
g\in \dom  (A_\Theta)
=\big\{ h \in \dom (S^*) \,\big|\,  \cB \Gamma_0 h= \cA \Gamma_1 h \big\}.
\end{equation}
In this case the formula \eqref{ttt1}
can be written as
\begin{align}\label{ttt1+}
\begin{split}
&  \st_\Theta[f,g]= \st_{S_1} [f,g]+\big(\Lambda f, \cA^{[-1]}\cB  \Lambda g\big)_{\bbC^n},  \\[1mm]
& f,g \in  \dom (\st_\Theta)=\big\{ h \in \dom (\st_{S_1}) \,\big|\, \Lambda h \in (\ker (\cA))^\perp  \big\}.
\end{split}
\end{align}
The expression \eqref{ttt1+} can be further simplified in the situations described in \eqref{tes4} and \eqref{tes00}.

\medskip
\noindent
\textbf{Two Limit Circle Endpoints.}
Return to the situation of Proposition \ref{23.t7.t3.12}. Then choose  the boundary triplet
$\{\bbC^2, \Gamma_0, \Gamma_1\}$, defined on $\dom (T_{max})$,  by
\begin{equation}\label{bt0}
  \Gamma_0 g=\begin{pmatrix} \widetilde g(a) \\ \widetilde g(b) \end{pmatrix}, \,\,\,
  \Gamma_1 g=\begin{pmatrix} \widetilde g^{\, \prime}(a) \\ -\widetilde g^{\, \prime} (b) \end{pmatrix},
  \quad g \in \dom (T_{max}).
\end{equation}
Furthermore, introduce the form  $\st[f,g]=\mathfrak Q_{c,d}(f,g), f,g \in \dom (\st)=\dom (\mathfrak Q_{c,d})$, 
as in \eqref{3.4} and \eqref{3.5};
cf. \cite[Equation (6.11.2)]{BHS20}. 
Then it is easy to see that
\begin{equation}\label{bt100}
 (f, T_{min} g)_{\Lr}=\st[f,g], \quad f \in \dom (\st), \,\, g \in \dom (T_{min}) \subseteq \dom (\st),  
\end{equation}
see \cite[Corollary 6.11.2]{BHS20}, and
\begin{equation}\label{bt101}
\dom (T_{max}) \subseteq \dom (\st),
\end{equation}
see \cite[Lemma 6.11.3]{BHS20}. Define $\Lambda : \dom (\st) \to \bbC^2$ by
\begin{equation}\label{bt1}
  \Lambda g=\begin{pmatrix} \widetilde g(a) \\ \widetilde g(b) \end{pmatrix},
  \quad g \in \dom (\st).
\end{equation}
For every $\varepsilon >0$  there exists $C_\varepsilon >0$ such that
\begin{equation}\label{bt103}
\|\Lambda g \|^2_{\bbC^2} \leq \varepsilon \st[g,g]+C_\varepsilon \|g\|^2_{\Lr}, \quad g \in \dom (\st),
\end{equation}
see \cite[Lemma 6.11.4]{BHS20}. It now follows
from \eqref{bt100}--\eqref{bt103}
that $\{\bbC^2, \Lambda\}$ is a boundary pair
which is compatible with the boundary triplet in \eqref{bt0}, see \cite[Lemma 6.11.5]{BHS20}.
Thus we can apply \eqref{ttt1}, see \cite[Theorem 6.11.6]{BHS20}.  
Note that $\mathfrak{Q}_{c,d} =\st_{S_1}$,
where $\dom (S_1)= \ker (\Gamma_1)$; cf. \eqref{bt0}. 
The self-adjoint extensions of $T_{min}$ are  parametrized via \eqref{bt000+}, 
given \eqref{tes2} and \eqref{tes1}. As before, our treatment will distinguish between separated
and coupled boundary conditions. 

\medskip

First, consider the case of separated boundary conditions in Theorem \ref{case0},  
 where $\cA$ and $\cB$ are $2\times 2$ matrices of the form
\begin{equation}\label{ww}
  \cA=\begin{pmatrix} - \sin (\alpha) & 0 \\   0 & \sin (\beta)  \end{pmatrix}, \quad
  \cB=\begin{pmatrix} \cos (\alpha)  & 0 \\ 0 & \cos (\beta)  \end{pmatrix}.
 \end{equation}
Note that \eqref{tes1} is satisfied.
 There are three subcases to be discussed.
First consider the case $\alpha \neq 0$ and $\beta \neq 0$.
Then $\cA$ is invertible and it follows  
\eqref{ww}
that $\Theta$   is given by
 \begin{equation}\label{zz01}
 \Theta =\cA^{-1} \cB= \begin{pmatrix} - \cot (\alpha) & 0 \\
 0 & \cot (\beta)  \end{pmatrix}.
\end{equation}
Substitution of \eqref{bt1} and \eqref{zz01} into \eqref{ttt1+} leads to \eqref{3.8} in Theorem~\ref{case0}.
The second case is that either $\alpha=0$ or $\beta=0$ (without equality simultaneously).
Assume $\alpha=0$.  Then $\ker (\cA)$ is one-dimensional  and, in fact,
it follows from $ \mul (\Theta)=\ker (\cA)$ and  $\dom (\Theta)=(\mul (\Theta))^\perp$
that
\begin{equation*}
 \mul (\Theta) =\ls \left( \begin{pmatrix} 1 \\  0 \end{pmatrix} \right),
 \quad
 \dom (\Theta)= \ls  \left( \begin{pmatrix} 0 \\ 1 \end{pmatrix} \right).
\end{equation*}
Therefore one sees from
\begin{equation*}
 \cB {\bf u}=\begin{pmatrix} 0 \\ \cos (\beta) \end{pmatrix},
 \quad
 \cA {\bf u}=\begin{pmatrix}  0 \\ \sin (\beta)  \end{pmatrix},   \quad
 {\bf u}=\begin{pmatrix} 0 \\ 1 \end{pmatrix},
\end{equation*}
together with \eqref{tes4},  that $c_\Theta=\cot (\beta)$ and hence the operator
$\Theta_{\rm op}$ acting in $\dom (\Theta)=\ls ({\bf u})$ is given by
 \begin{equation*}
 \Theta_{\rm op}=\cot (\beta),
 \quad \dom (\st_\Theta)= \big\{ h \in \dom (\st) \,\big|\, \widetilde h(a)=0\big\}.
\end{equation*}
This together with \eqref{ttt1+} and \eqref{tes4} leads to \eqref{3.10a}.
Likewise, when $\beta=0$, then $c_\Theta=-\cot(\alpha)$ and hence
\begin{equation*}
 \Theta_{\rm op}=-\cot (\alpha),
 \quad \dom (\st_\Theta)=\{ h \in \dom (\st) \,|\, \widetilde h(b)=0 \},
\end{equation*}
and this leads to \eqref{3.20aa}.
The third case concerns $\alpha=\beta=0$.
Then $\mul (\Theta)=\ker (\cA)=\bbC^2$ and $\dom (\Theta)=\{0\}$.
Thus $\Theta_{\rm op}$ is trivial and
\begin{equation*}
\st_\Theta \subseteq \st,  \quad
\dom (\st_\Theta)= \big\{ h \in \dom (\st) \,\big|\, \widetilde h(a)=0=\widetilde h(b)\big\},
\end{equation*}
see \eqref{3.22aa}, which corresponds to the Friedrichs extension. This treats
all cases of Theorem \ref{case0}.

\medskip

Secondly, consider the case of coupled boundary conditions in Theorem \ref{case1}, 
 where $\cA$ and $\cB$ are $2\times 2$ matrices of the form 
\begin{equation}\label{yy}
  \cA=-\begin{pmatrix} e^{i \varphi} R_{1,2} & 0 \\ e^{i \varphi} R_{2,2} & 1 \end{pmatrix},
\quad
  \cB=\begin{pmatrix} e^{i \varphi} R_{1,1} & -1 \\ e^{i \varphi} R_{2,1} & 0 \end{pmatrix},
 \end{equation}
and hence \eqref{tes1} is satisfied.
 There are two subcases to be discussed.

The first subcase is when  $R_{1,2} \neq 0$.
 Then $\cA$ is invertible and it follows from  
 \eqref{yy} and $\det_{\bbC^2}(R)=1$,
 that $\Theta$  is given by
\begin{equation}\label{zz0}
  \Theta =\cA^{-1} \cB= - \frac{1}{R_{1,2}} \begin{pmatrix} R_{1,1} & -e^{-i \varphi} \\
 -e^{i \varphi} & R_{2,2} \end{pmatrix}.
\end{equation}
It follows from \eqref{ttt1} and the expression in \eqref{zz0} that
\begin{equation*}
(\Lambda f, \Theta \Lambda g)_{\bbC^2}=- \frac{1}{R_{1,2}}
 \begin{pmatrix} \widetilde f(a) \\ \widetilde f(b) \end{pmatrix}^*
 \begin{pmatrix} R_{1,1} & -e^{-i \varphi} \\ -e^{i \varphi} & R_{2,2} \end{pmatrix}
 \begin{pmatrix} \widetilde g(a) \\ \widetilde g(b) \end{pmatrix}, \quad f,g \in \dom (\st). 
\end{equation*}
Together with \eqref{ttt1+} this implies Theorem \ref{case1}\,$(i)$.   

The second subcase occurs when $R_{1,2}=0$,
which implies that $1=\det_{\bbC^2}(R)=R_{1,1}R_{2,2}$.
Then $\ker (\cA)$ is one-dimensional  and, in fact, it follows from $ \mul (\Theta)=\ker (\cA)$ and  $\dom (\Theta)=(\mul (\Theta))^\perp$
that
\begin{equation*}
 \mul (\Theta) =\ls \left(\begin{pmatrix} 1 \\  -e^{i\varphi} R_{2,2} \end{pmatrix} \right),
 \quad
 \dom (\Theta)= \ls  \left( \begin{pmatrix} e^{-i \varphi} R_{2,2} \\ 1 \end{pmatrix} \right).
\end{equation*}
Therefore one sees from
\begin{equation*}
 \cB {\bf u}=\begin{pmatrix} R_{1,1}R_{2,2} -1 \\ R_{2,1}R_{2,2} \end{pmatrix},
 \quad
 \cA {\bf u}=\begin{pmatrix}  0 \\ -R_{2,2}^2-1  \end{pmatrix},  \quad
 {\bf u}=\begin{pmatrix} e^{-i \varphi} R_{2,2} \\ 1 \end{pmatrix},
\end{equation*}
together with \eqref{tes4},  that
 \begin{equation}\label{zz1}
  c_\Theta=-\frac{R_{2,1}}{R_{1,1}+R_{2,2}}.
\end{equation}
Thus by \eqref{ttt1} and the expression in \eqref{zz1} it is clear that
\begin{align*}
\begin{split}
& (\Lambda f, \Theta_{\rm op} \Lambda g)_{\bbC^2}= -R_{1,1}R_{2,1} \overline{\widetilde f(a)} \widetilde g(a),  \\[1mm]
& f,g \in \dom (\st_\Theta)= \big\{ h \in \dom (\st) \,\big|\, \widetilde h(b)=e^{i\varphi} R_{1,1} \widetilde h(a)\big\}.
\end{split}
\end{align*}
Together with \eqref{ttt1+} this implies Theorem \ref{case1}\,$(ii)$.

\medskip
\noindent
\textbf{One Limit Circle Endpoint.} Return to  the situation of Proposition \ref{p4.1}.
Choose the boundary triplet $\{\bbC, \Gamma_0, \Gamma_1\}$, defined on $\dom (T_{max})$, by
\begin{equation}\label{lbt3}
  \Gamma_0 g= \widetilde g(a), \quad  \Gamma_1 g= \widetilde g^{\, \prime}(a),
  \quad g \in \dom (T_{max}).
\end{equation}
Furthermore, introduce the form  $\st[f,g]=Q_{c,d}(f,g)$, $f,g \in \dom (\st)=\dom (Q_{c,d})$, as in \eqref{4.3} and \eqref{4.4};
see \cite[Equation (6.12.2)]{BHS20}. 
Then it is easy to see that
\begin{equation}\label{lbt100}
 (f, T_{min} g)_{\sH}=\st[f,g], \quad  f \in \dom (\st), \,\,  g \in \dom (T_{min}) \subseteq\dom (\st),
\end{equation}
see \cite[Corollary 6.12.2]{BHS20}, and
\begin{equation}\label{lbt101}
\dom (T_{max}) \subseteq \dom (\st),
\end{equation}
see \cite[Lemma 6.12.3]{BHS20}.
Define $\Lambda : \dom (\st) \to \bbC$ by
\begin{equation}\label{lbt1}
  \Lambda g=\widetilde g(a),
  \quad g \in \dom (\st).
\end{equation}
For every $\varepsilon >0$  there exists $C_\varepsilon >0$ such that
\begin{equation}\label{lbt103}
\|\Lambda g \|^2_{\bbC} \leq \varepsilon \st[g,g]+C_\varepsilon \|g\|^2_{\Lr}, \quad g \in \dom (\st),
\end{equation}
see \cite[Lemma 6.12.4]{BHS20}. It now follows
from \eqref{lbt100}--\eqref{lbt103}
that $\{\bbC, \Lambda\}$ is a boundary pair
which is compatible with the boundary triplet in \eqref{lbt3}, see \cite[Lemma 6.12.5]{BHS20}.
Thus we can apply \eqref{ttt1}, see \cite[Theorem 6.12.6]{BHS20}.
Note that $\mathfrak{Q}_{c,d} =\st_{S_1}$,
where $\dom (S_1)=\ker (\Gamma_1)$; see \eqref{lbt3}. 
 
\medskip 
 
 The self-adjoint extensions of $T_{min}$ are now parametrized via \eqref{tes00}
\begin{equation*}
  \cos (\alpha) \Gamma_0 g + \sin (\alpha)  \Gamma_1 g =0, \quad g \in \dom (T_{max}),
\end{equation*}
over $\alpha \in [0, \pi)$, and denoted by $T_\alpha$, see Proposition \ref{p4.1}.
 Therefore, one has for $\alpha \in (0, \pi)$, that
 \begin{equation*}
\st_\alpha[f,g]= \st[f,g] - \cot (\alpha) (\Lambda f, \Lambda g)_{\bbC},
\quad f,g \in \dom (\st_\alpha)=\dom (\st).
\end{equation*}
Moreover, if $\alpha=0$, then
\begin{equation*}
\st_\alpha \subseteq \st, \quad
\dom (\st_\alpha)=\big\{ h \in \dom (\st) \,\big|\, \widetilde h(a)=0\big\},
\end{equation*}
which corresponds to the Friedrichs extension.
This implies Theorem \ref{caselp}, cf. \cite[Theorem 6.12.6]{BHS20}.

\medskip

For a succinct treatment of boundary triplets and Weyl--Titchmarsh functions tailored towards ordinary differential operators (a.k.a., ``boundary triplets in a nutshell''), see also \cite[App.~D.7]{GNZ24}. Likewise, a treatment of boundary pairs, going back to \cite{A96}, can be found in \cite[Ch.~5]{BHS20}.

\medskip


\noindent
{\bf Acknowledgments.}
J.B.\ is most grateful for a stimulating research stay at Baylor University, where some parts of this paper were written in May of 2025. F.G.\ and H.S.\ gratefully acknowledge kind invitations to the Institute of Applied Mathematics at the Graz University of Technology, Austria. 
This research was funded by the Austrian Science Fund (FWF)
Grant-DOI: 10.55776/P33568.

%


\end{document}